\documentclass[11pt]{amsart}
\usepackage{amsmath}
\usepackage{graphicx}

\newtheorem{proposition}{Proposition}[section]
\newtheorem{theorem}[proposition]{Theorem}
\newtheorem{lemma}[proposition]{Lemma}
\newtheorem{corollary}[proposition]{Corollary}

\theoremstyle{definition}
\newtheorem{definition}[proposition]{Definition}

\theoremstyle{remark}
\newtheorem{remark}[proposition]{Remark}
\newtheorem{example}[proposition]{Example}

\def\C{\mathcal{C}}
\def\R{\mathbb{R}}
\def\Z{\mathbb{Z}}
\def\fig#1{\raisebox{-2.2ex}{\includegraphics[height=5.2ex]{#1}}}
\def\figg#1{\raisebox{-3.2ex}{\includegraphics[height=8.2ex]{#1}}}

\begin{document}

\title[Normal rulings of solid torus links]{HOMFLY-PT polynomial and normal rulings of Legendrian solid torus links}
\author{Dan Rutherford}

\maketitle

\begin{abstract} We show that for any Legendrian link $L$ in the $1$-jet space of $S^1$ the $2$-graded ruling polynomial, $R^2_L(z)$, is determined by the Thurston-Bennequin number and the HOMFLY-PT polynomial.  Specifically, we recover $R^2_L(z)$ as a coefficient of a particular specialization of the HOMFLY-PT polynomial.  Furthermore, we show that this specialization may be interpreted as the
standard inner product on the algebra of symmetric functions that is often
identified with a certain subalgebra of the HOMFLY-PT skein module of the
solid torus.


In contrast to the $2$-graded case, we are able to use $0$-graded ruling polynomials to distinguish many homotopically non-trivial Legendrian links with identical classical invariants.  

\end{abstract}

\section{Introduction}

The study of Legendrian knots in standard contact $\R^3$ up to the equivalence relation of Legendrian isotopy provides an interesting  variation on the classical theory of smooth knots in $3$-space.  Each smooth knot type has Legendrian representatives.  However, Legendrian knots of the same underlying smooth knot type need not be equivalent as Legendrian knots.  

There are two `classical invariants' capable of distinguishing between Legendrian knots with the same underlying smooth knot type.  They are known as the Thurston-Bennequin number, $\textit{tb}(L)$,  and rotation number, $r(L)$.  Beginning in the late 1990's, several stronger invariants of Legendrian knots have been developed.  Of particular interest for this article are invariants arising from counts of certain decompositions of front diagrams known as normal rulings.  Normal rulings arose independently in the work of Fuchs \cite{F} 
in connection with augmentations of the Chekanov-Eliashberg DGA and also in the work of Chekanov and Pushkar \cite{ChP} who were motivated by generating families.  Chekanov and Pushkar defined for each divisor $p$ of $2 r(L)$ an invariant which can be neatly encoded as the $p$-graded ruling polynomial, $R^p_L(z)$.

A particularly elegant aspect of Legendrian knot theory is the interplay between Legendrian invariants and invariants of the underlying smooth knot type.  
For instance, the values of the classical invariants are constrained by invariants which depend only on the smooth knot type via `Bennequin type inequalities' (see for instance \cite{Ng}).  As an example, Fuchs and Tabachnikov \cite{FT} proved that for any Legendrian link $L$ in standard contact $\R^3$
\begin{equation}
\label{eq:FuchsTab}
\textit{tb}(L) + |r(L)| \leq - \deg_aP_L(a, z)
\end{equation}
where $P_L \in \Z[a^{\pm 1}, z^{\pm 1}]$ is the HOMFLY-PT polynomial\footnote{For consistency with the main body of this article $P_L$ is normalized so that the unknot has the value $(a- a^{-1})/z$. This differs from \cite{FT} and \cite{R}.}.
In turn, Legendrian knots can shed light on topological knot invariants.  It is shown in \cite{R} that the coefficient of $a^{-\textit{tb}(L) }$ in $P_L(a,z)$ is precisely $R^2_L(z)$, and hence may be viewed as counting $2$-graded normal rulings.  

The main purpose of the present article is to investigate the relationship between $2$-graded normal rulings and the HOMFLY-PT polynomial of Legendrian links in the $1$-jet space of the circle, $J^1(S^1)$.  $J^1(S^1)$ is a contact manifold diffeomorphic to an open solid torus, and Legendrian links in $J^1(S^1)$ can be represented diagrammatically via their front projections to the annulus.  

The solid torus case is more interesting than $\R^3$ due to the nature of the HOMFLY-PT polynomial.  Unlike link diagrams in the plane, annular link diagrams cannot always be reduced to a multiple of the unknot via repeated applications of the skein relations.  Instead there are sequences of oriented diagrams $A_{\pm 1}, A_{\pm 2}, \ldots$ whose products (defined by stacking, see Section 2) form the base cases for evaluating $P_L$.  This gives rise to a HOMFLY-PT polynomial with many new variables,
\[
P_L \in \Z[a^{\pm 1}, z^{\pm 1}, A_{\pm 1}, A_{\pm 2}, \ldots].
\]
More systematically, one considers the skein module, $\C$, obtained by imposing the HOMFLY-PT relations on formal linear combinations of link diagrams.  Turaev \cite{Tu1} showed that $\C$ is a free module with linear basis consisting of products of the $A_i$. For a given monomial $A_{i_1}\cdots A_{i_N}$ we collect the terms with positive and negative indices to write   $A_{i_1}\cdots A_{i_N} = A_\lambda A_{-\mu}$ for partitions $\lambda$ and $\mu$ (see Section 4).   $P_L$ is simply a normalization of the expansion of $L$ in Turaev's basis, $\{A_\lambda A_{-\mu}\}$. Chmutov and Goryunov \cite{CG} extended the estimate (\ref{eq:FuchsTab}) to the $J^1(S^1)$ setting (see Theorem~\ref{chmutov}).  

One of our main results is the following:

\medskip

\noindent{\bf Theorem~\ref{mainT}.} \quad {\it   For any Legendrian link $L \subset J^1(S^1)$, 
\[
R^2_L(z) = \mbox{coefficient of }\, a^{-\textit{tb}(L)} \, \, \mbox{in} \, \, \widehat{P}_L(a,z). 
\]
}

Here,  $\widehat{P}_L(a,z)$ is obtained from the HOMFLY-PT polynomial by specializing the variables $A_i$ in a non-multiplicative manner as
\[
A_{\lambda}A_{ -\mu} \mapsto \langle A_\lambda, A_\mu \rangle := R^2_{A_\lambda A_{-\mu}}(z) \in \Z[a^{\pm 1}, z^{\pm 1}].
\]
$\langle \cdot, \cdot \rangle$ may be viewed as a bilinear form on the subalgebra $\C^+ \subset \C$  generated by $A_i$ with $i>0$.  Theorem~\ref{the:RulComp2} gives a computation of $\langle A_\lambda A_{-\mu} \rangle$ in terms of the partitions $\lambda$ and $\mu$ as a sum over the class of non-negative integer entry matrices with row sum $\lambda$ and column sum $\mu$.  The summands are determined by the entries of the matrix.  

In the literature, there is a traditional way of identifying $\C^+$ with the algebra, $\Lambda$, of symmetric functions (see \cite{AM}, \cite{Lu}, \cite{MM} and the discussion in Section 5.1) where the Schur functions $s_\lambda$ correspond to skein elements $Q_\lambda$.  $\Lambda$ has a standard inner product arising from taking the Schur functions as an orthonormal basis. In Section 5,  we show that the corresponding inner product, $(Q_\lambda, Q_\mu) = \delta_{\lambda, \mu}$, agrees with the bilinear form $\langle \cdot, \cdot \rangle$ used to define $\widehat{P}_L(a,z)$.

\medskip

\noindent{\bf Theorem~\ref{the:product}.} \quad {\it For any partitions $\lambda$ and $\mu$,}
\[
(A_\lambda, A_\mu) = R^2_{A_{\lambda}A_{-\mu}}(z) = \langle A_\lambda, A_\mu \rangle.
\]
\noindent Thus, with respect to Turaev's basis $A_\lambda$ the inner product on $\Lambda$ has a skein theoretic interpretation using Legendrian links.

Our method for proving Theorem~\ref{mainT} is inductive as in \cite{R}, but several interesting complications arise.  For starters, a more subtle measure of the complexity of a front diagram is required and an additional algorithm is necessary to reduce the complexity of front diagrams lacking cusps.  More notably, the base case for the induction needs to be enlarged to include all products of Legendrian versions of the $A_i$.  An interesting wrinkle occurs here.  In contrast to the case of smooth link diagrams, the product of Legendrian front diagrams in the annulus is not commutative.  This phenomenon was first observed by Traynor who showed that the two components of the Legendrian link $L = L_0 \sqcup L_1$ cannot be interchanged via a Legendrian isotopy.  Here, $L_0$ and $L_1$ denote the $1$-jets of the constant functions $0$ and $1$ on $S^1$.  In Theorem~\ref{NonCommute} we provide many further examples by showing that for any $i, j \in \Z \setminus \{0\}$ the locations of disjoint $A_i$ and $A_j$ in the $z$ direction cannot be interchanged by a Legendrian isotopy.
Nevertheless, we are able to establish in Lemma~\ref{lem:Commute} that the $2$-graded ruling polynomial of a product of the $A_i$ does not depend on the ordering of the factors.  

We've included at the end of Section 6 a proof of Chmutov and Goryunov's estimate  (Theorem~\ref{chmutov}).  The HOMFLY-PT polynomials used here and in \cite{CG} differ in a significant way (see Section 6.2).  While we believe that these two versions of $P_L$ should provide the same estimate for $\textit{tb}(L) + |r(L)|$, it is straight forward to provide a proof of Theorem~\ref{chmutov} from scratch. Our proof is based on the inductive method used in the proof of Theorem~\ref{mainT} and is similar in spirit to Ng's approach to Bennequin type inequalities in $\R^3$ \cite{Ng}.

\subsection{Acknowledgements}

I would like to thank Lenny Ng for many useful discussions during the course of this project.  Also, I am grateful to AIM for hosting a workshop on `Legendrian and transverse knots' as well as a follow up SquaREs on `Augmentations, rulings, and generating families'.   I thank fellow SquaREs participants Dmitry Fuchs, Brad Henry, Paul Melvin, Josh Sabloff, and Lisa Traynor for many discussions regarding normal rulings and related topics.  In addition, I thank Greg Kuperberg for a useful conversation about skein modules.

\section{Legendrian Links in $J^1(S^1)$}

The $1$-jet space of $S^1$, 
\[\displaystyle
J^1(S^1)  =  T^*S^1 \times \R = \{ (x,y,z) | x \in S^1,\, y,z \in \R\},
\]
is diffeomorphic to an open solid torus and is equipped with the contact structure $\xi = \mbox{ker}( dz - y \, dx )$.    A smooth (oriented) link $L \subset J^1(S^1)$ is called {\it Legendrian } if it is everywhere tangent to $\xi$.  Two Legendrian links are {\it Legendrian isotopic} if they are isotopic through other Legendrian links.  A Legendrian link $L$ is determined by its {\it front projection} (also denoted $L$) to the annulus,
\[J^1(S^1) \rightarrow S^1 \times \R   , \, (x,y,z) \mapsto (x,z),\]
because the $y$-coordinate of $L$ is recovered as the slope $dz/dx$.  Viewing $S^1$ as $[0, 1]/ \{0, 1\}$, we visualize the front projection of $L$ as a collection of arcs in $[0, 1] \times \R$ with identifications at the boundary. 

A front projection of a Legendrian link is called {\it generic} if it is immersed away from semi-cubical cusp points and the only self intersections are transverse double points.  $L$ is called {\it $\sigma$-generic} if in addition the double points and cusps all have distinct $x$-coordinates.  Any  collection of closed curves in the annulus without vertical tangencies and satisfying the conditions of a generic front projection may be lifted to a unique Legendrian link in $J^1(S^1)$.  It is not necessary to indicate the over/under relationship between two strands at a crossing of a front projection.  The $y$-axis is oriented away from the observer, so the strand with lesser slope always appears on top.  See Figure \ref{fig:FrontExample} below for an example of a front projection.  

The equivalence relation of Legendrian isotopy may be formulated in a somewhat combinatorial fashion using front projections \cite{Sw}.  Any Legendrian isotopy class has representatives with generic front projections.  Furthermore, if two generic front projections represent Legendrian isotopic links then one may be transformed into the other via a combination of  the {\it Legendrian Reidemeister moves} indicated in Figure ~\ref{fig:LRMoves} and isotopies of the plane which do not introduce vertical tangencies.

\begin{figure}
\centerline{\includegraphics[scale=.6]{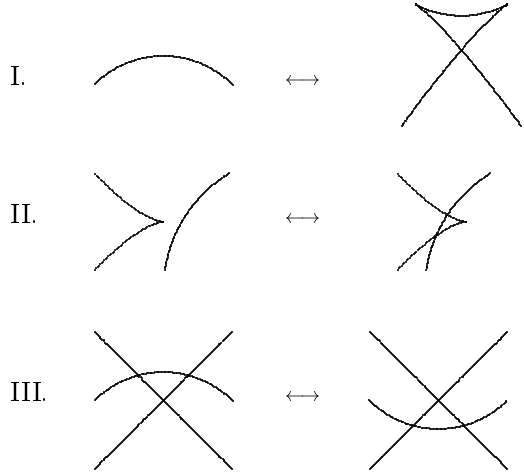}}
\caption{Legendrian Reidemeister moves.}
\label{fig:LRMoves}
\end{figure}

\subsection{Product of fronts} Given front projections $K$ and $L$ we define their product $K\cdot L$ by stacking $K$ vertically above $L$ in $S^1 \times \R$,  
\[
K\cdot L = \raisebox{-7.2ex}{\includegraphics[height=16.2ex]{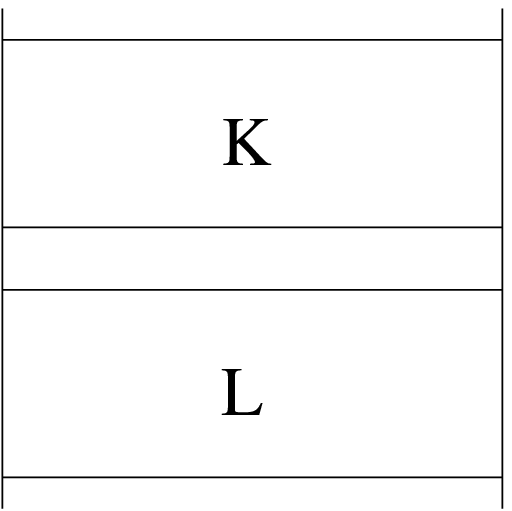}}.
\]
This product is well defined on Legendrian isotopy classes.

\begin{remark}  The corresponding product on smooth knot types is commutative.  However, it follows from work of Traynor  \cite{T1}  that this is not always the case for Legendrian knot types.  See subsection 4.1.
\end{remark}

\subsection{Classical invariants}

The simplest invariants capable of distinguishing between Legendrian links with the same underlying smooth link type are the  {\it Thurston-Bennequin number}, $\textit{tb}$, and the {\it rotation number}, $r$.  For a Legendrian link $L \subset J^1(S^1)$ the Thurston-Bennequin number is computed from a front projection of $L$ as 
\[ \textit{tb}(L) = \mbox{writhe}(L) - \frac{1}{2} (\# \, \mbox{of cusps}),
\]
and the rotation number is given by 
\[
r(L) = \frac{1}{2} ( (\# \, \mbox{of downward oriented cusps}) -  (\# \, \mbox{of upward oriented cusps})).
\]  

\begin{remark}  If $L$ is homologically trival then $\textit{tb}(L)$ is the linking number of $L$ with a link $L^+$ obtained by a small shift in the oriented normal direction to the contact planes.  There exist differing conventions for extending the definition of $\textit{tb}$ to homologically non-trivial links in $J^1(S^1)$.  We follow the definition used in \cite{NgTr} which is natural when working with front projections.  Geometrically, $\textit{tb}(L)$ is the index of intersection of $L^+$ with an oriented surface bounded by $L$ and an appropriate number of copies of $A_1$ or $A_{-1}$ (see Section 4) located far away from $L$ in the $z$-direction.  Alternatively, Tabachnikov defined in \cite{Ta} a ``Bennequin affine invariant'' for Legendrian links in $ST^*\R^2$ using instead an oriented surface bounded by $L$ and some number of distant fibers of the projection $ST^*\R^2 \rightarrow \R^2$.  Under the standard contactomorphism $ ST^*\R^2 \cong J^1(S^1) $ the front diagrams of these fibers appear as phase shifted cosine functions with large amplitudes. \cite{CG} follows the latter convention.
\end{remark}

\section{Normal rulings in $J^1(S^1)$} 

In this section we review Legendrian isotopy invariants introduced by Chekanov and Pushkar in \cite{ChP}.  The invariants depend on a choice of divisor $p | 2 r(L)$ and are computed as counts of additional combinatorial structures associated to a front diagram which we will call $p$-graded normal rulings.  This terminology follows Fuchs \cite{F} who, in connection with augmentations of the Chekanov-Eliashberg DGA, independently introduced similar combinatorial structures for front diagrams of Legendrian knots in standard contact $\R^3$.

\subsection{Maslov Potentials}

After removing cusp points a front diagram is divided into a union of immersed curves which we will call {\it strands}.  Note that along each strand the orientation of $L$ either entirely agrees or entirely disagrees with the orientation of $S^1 = [0, 1]/\{0,1\}$.  In this regard we may view the orientation of $L$ as a function from the strands of $L$ to $\Z/2\Z$ where we make the convention that the value $0$ ($1$) indicates a strand oriented to the right (left).   

\begin{definition} A {\it Maslov potential} for a generic front diagram of a Legendrian knot $L$ is a function $\mu$ from the strands of $L$ to $\Z / (2 r(L) \Z) $  so that at cusps the value increases by $1$ when moving from the lower half of the cusp to the upper half.  See Figure ~\ref{fig:Mas}.  For convenience, we require that reducing $\mu$ $ \mod 2$ gives the orientation of $L$.  A Maslov potential for a multi-component link is a choice of Maslov potential for each component.
\end{definition}

Maslov potentials are extended in an obvious way along a generic Legendrian isotopy.

\begin{figure}
\centerline{\includegraphics[scale=1]{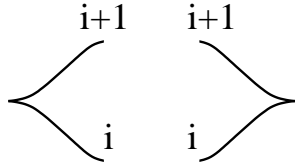}}
\caption{A Maslov potential near cusps.}
\label{fig:Mas}
\end{figure}

\subsection{ $p$-graded normal rulings }

  Suppose that $L$ is a $\sigma$-generic front projection.  Under this assumption the subset $\Sigma \subset S^1$ of $x$-values where $L$ has double points or cusps is finite and for each $x_0 \in \Sigma$ the subset $\{x= x_0 \} \subset S^1 \times \R$ intersects a single crossing or cusp of $L$. 
  
 Let $\pi : S^1\times \R \rightarrow S^1$ denote the projection and, for each $x \in S^1$,  $L_x = L \cap \pi^{-1}(x)$. 

\begin{definition} A continuous function $f$ from a subset $N \subset S^1$ to the front projection $L \subset S^1 \times \R$ is called a {\it section} if $\pi \circ f = \mbox{id}_N$.  
\end{definition}

\begin{definition}  A { \it normal ruling} of a front projection $L$ is a continuous function  $\rho : L \setminus \pi^{-1}(\Sigma) \rightarrow L \setminus \pi^{-1}(\Sigma) $ satisfying:
\begin{enumerate}

\item  $\pi \circ \rho = \pi|_{L \setminus \pi^{-1}(\Sigma)}$, so for each $x \in S^1 \setminus \Sigma$ there is a restriction $\rho_x : L_x \rightarrow L_x$.  

\item  Each $\rho_x$ is a fixed point free involution.  

This condition together with the continuity of $\rho$ implies that on any interval of $S^1 \setminus \Sigma$ the strands of $L$ are divided into pairs.  The remaining requirements give restrictions on this pairing near crossings and cusps.

\item  Strands meeting at a cusp are paired by the involutions $\rho_x$ in a neighborhood of the cusp point.  The pairing of the remaining non-cusp strands should agree before and after the cusp.  

\item   Near a crossing the two strands that meet should not be paired together by $\rho_x$.  

\item   The pairing of strands arising from $\rho_x$ can be continuously extended along a crossing in the following sense.  Let  $x_0 \in \Sigma$ such that $L_{x_0}$ contains a crossing of $L$.  In a neighborhood $N \subset S^1$ of  $x_0$ one should be able to find a number of sections $f_1,  \dots , f_n :N \rightarrow L$ so that every point of $L\cap \pi^{-1}(N)$ is in the image of exactly one of the $f_i$ with the exception of the double point which is in the image of two of the $f_i$.  Furthermore, these sections should be preserved by the involutions $\rho_x$, so that
on $N\setminus \Sigma$ for each $f_i$, $\rho \circ f_i = f_j$ for some $f_j, 1 \leq j \leq n$.  For the  two sections meeting at the crossing there are two possiblities.  Either they follow the diagram and cross transversally at the crossing, or they switch strands at the crossing in a non-smooth manner. In the latter case, the crossing is called a {\it switch} of $\rho$.  

Finally, we have a restriction at switches known as the {\it normality condition}.

\item   Near switches of $\rho$ the two intervals on the $z$-axis arising from connecting crossing strands to their companion strands are either disjoint or one is contained in the other.
\end{enumerate}
\end{definition}

\begin{remark}  Three of the six possible arrangements of the switching strands and their companions along the vertical axis are prohibited by the normality condition.  See Figure ~\ref{fig:NormC}.
\end{remark}

\begin{figure}
\centerline{\includegraphics[scale=1]{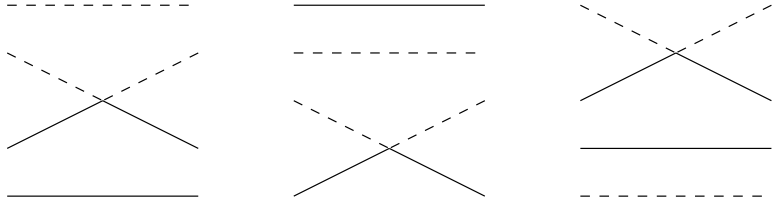}}
\caption{The three configurations for switching strands and companion strands allowed by the normality condition.}
\label{fig:NormC}
\end{figure}

Suppose that a Legendrian link $L$ has components $L_1, \ldots, L_N$, and let $p$ be a common divisor of $2r(L_i), i = 1, \ldots, N$.  

\begin{definition}
A normal ruling $\rho$ of $L$ is called {\it $p$-graded} with respect to a particular Maslov potential $\mu$ for $L$ if, after reducing $\mu$ modulo $p$, whenever two strands are paired by the involutions $\rho_x$ the strand with the larger $z$-coordinate has Maslov potential $1$ larger than the strand with smaller $z$-coordinate.  That is,
\[
\rho(x,z) = (x, z') \, \, \mbox{and} \,\, z' > z \Rightarrow \mu(x,z') = \mu(x,z) +1 \mod p.
\]
\end{definition}

\begin{remark}  (i) Every normal ruling is $1$-graded.

(ii)  We are most interested in the case $p = 2$.  Note that, a normal ruling is $2$-graded exactly when $\rho$ reverses orientation.  Choosing a Maslov potential is unnecessary.

(iii) If a normal ruling is $p$-graded, then at each of the switches the Maslov potentials of the crossing strands must agree modulo $p$.  However, in contrast to Legendrian links in $\R^3$ this condition is no longer sufficient for a normal ruling to be $p$-graded.  

(iv)  For a single component link the $p$-graded condition is independent of the choice of Maslov potential since any two Maslov potentials will differ by a constant.

(v) If $p$ is even, then the involutions $\rho_x$ reverse the orientation of $L$.  It follows that only null-homologous links can have $p$-graded normal rulings when $p$ is even.  

\end{remark}

Given a Legendrian link $L$ with $\sigma$-generic front projection and chosen Maslov potential $\mu$, we let $\Gamma^p( L, \mu)$ denote the set of normal rulings of $L$ which are $p$-graded with respect to $\mu$.  To each $\rho \in  \Gamma^p( L,\mu)$ we associate the integer\footnote{This differs by $1$ from the convention for $j(\rho)$ used in \cite{R}.}
\[
j(\rho) := \#( \mbox{switches} ) - \#( \mbox{right cusps} ).
\]
Finally, we define the {\it $p$-graded ruling polynomial}, $R^p_{L,\mu}(z)$, as
\[
R^p_{L,\mu}(z) =  \sum_{\rho \in \Gamma^p(L,\mu)} z^{j(\rho)}.
\]

As remarked above, if $p=1, 2$ or $L$ has a single component, then the choice of $\mu$ is not relevant and will be suppressed from the notation.

Given a sufficiently generic Legendrian isotopy between links $L_1$ and $L_2$ with $\sigma$-generic front projections, Chekanov and Pushkar provide a bijection between $ \Gamma^1(L_1)$ and $\Gamma^1(L_2)$ which preserves the integers $j(\rho)$.  Assuming the isotopy takes a Maslov potential $\mu_1$ for $L_1$ to the corresponding Maslov potential $\mu_2$ for $L_2$ their bijection takes $\Gamma^p(L_1, \mu_1)$ to $\Gamma^p(L_2, \mu_2)$.  

\begin{theorem}[\cite{ChP}] \label{RP}
If there is a Legendrian isotopy between $L_1$ and $L_2$ which is compatible with corresponding Maslov potentials $\mu_1$ and $\mu_2$ then
\[
R^p_{L_1,\mu_1}(z) = R^p_{L_2,\mu_2}(z).
\]
In particular, $R^1_L(z)$ and $R^2_L(z)$ are Legendrian isotopy invariants.
\end{theorem}

\begin{example} 

\begin{figure}
\centerline{\includegraphics[height=2in]{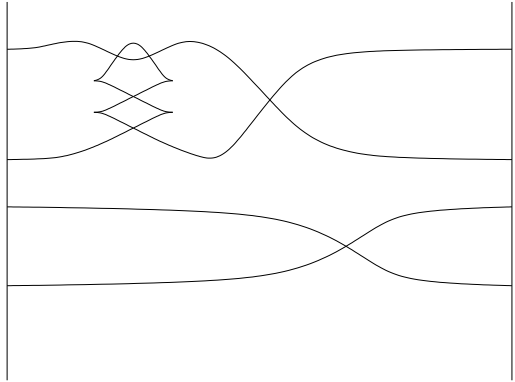}}
\caption{An annular front projection for a Legendrian link $K \subset J^1(S^1)$. } 
\label{fig:FrontExample}
\end{figure}

\begin{figure}
\centerline{\includegraphics[height=2in]{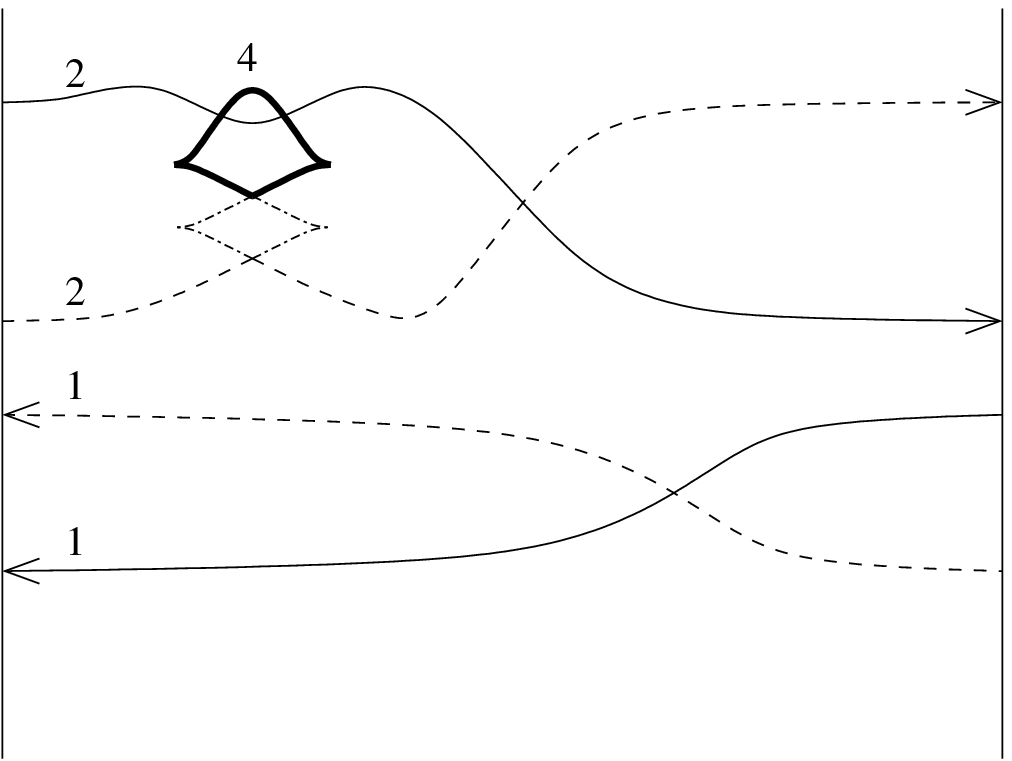} \,\,\,
\includegraphics[height=2in]{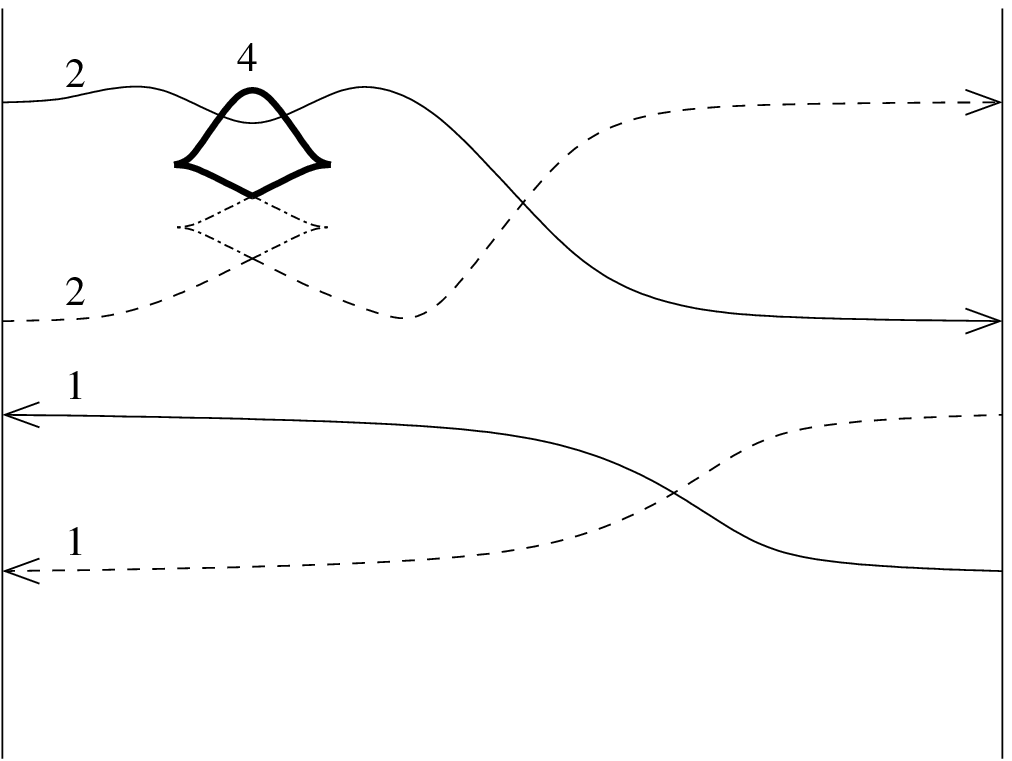}}
\caption{Two normal rulings of the front projection $K$. } 
\label{fig:RulExample}
\end{figure}

In Figure \ref{fig:FrontExample} a Legendrian link $K \subset J^1(S^1) $ is presented via its front projection to $S^1\times \R$.  Two normal rulings of $K$ are pictured in Figure \ref{fig:RulExample}.  Both of the rulings are $0$-graded with respect to the indicated Maslov potential, $\mu$.  In case the value of $\mu$ on the lower component were altered to $3$ the pictured rulings would remain $2$-graded but would fail to be $0$-graded.  $K$ has several other normal rulings, and  its $2$-graded and $0$-graded ruling polynomials are given by
\[
R^2_K(z) = 2 + 3 z^2 + z^4, \, \,\, R^0_{K, \mu}(z) = 2 + z^2.
\]
\end{example}

\section{Computation of $R^2$ for products of basic fronts}

For each positive integer $m \geq 1$, we consider the front diagram $A_m$ which consists of a single component wrapping $m $ times around the annulus with $m-1$ crossings.  $A_m$ is everywhere oriented to the right in $S^1 \times \R$ and can be viewed as the closure of the $m$-braid, $\sigma_1 \sigma_2 \ldots \sigma_{m-1}$.  Here, we compose braids from left to right and number strands from top to bottom.  See Figure ~\ref{fig:BasisKnots}.  We let $A_{-m}$ denote $A_m$ with its orientation reversed.

\begin{figure}
\centerline{
\includegraphics[scale=.8]{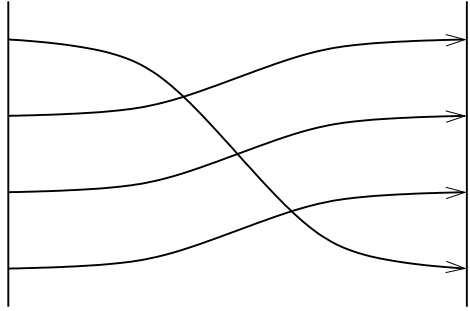}
}
\caption{The front diagram $A_4$.
}
\label{fig:BasisKnots}
\end{figure}

We will sometimes refer to the front diagrams $A_m$ as {\it basic fronts}.  The basic fronts will play a crucial role as their products form a basis for the HOMFLY-PT skein module of the annulus (See Section 5).  

Recall that a finite non-increasing sequence of positive integers $\lambda = (\lambda_1, \ldots, \lambda_\ell),$ $ \lambda_i  \geq \lambda_{i+1}, 1 \leq  i \leq \ell-1$ is called a {\it partition}.  If $\sum \lambda_i = n$ we say that $\lambda$ is a partition of $n$ and write $\lambda \vdash n$.  The integers $\lambda_i, 1 \leq i \leq \ell$ are called the parts of $\lambda$, and we write $\lambda = 1^{m_1} 2^{m_2}  \cdots r^{m_r}$ to indicate that $\lambda$ is the partition with $m_k$ parts equal to $k$, $1 \leq k \leq r$, and no part larger than $r$.   The total number of parts, $\ell= \ell(\lambda) = m_1 + \cdots + m_r$, is called the length of $\lambda$.  

\begin{lemma}
\label{lem:RulComp1}
For any $m \geq 1$, 
$$R^2_{A_m A_{-m} } (z) =  \sum_{\lambda \vdash m} \left ( \frac{\ell(\lambda)!}{m_1!m_2!\cdots m_r! } \right ) ( 1^{m_1} 2^{m_2} \cdots r^{m_r}) z^{2(\ell(\lambda) -1)}. $$
\end{lemma}

\begin{proof}

First, note that there are precisely $m$ normal rulings of $A_m A_{-m}$ which have no switches.  The continuity conditions required in the definition of normal ruling show that such a ruling is uniquely determined by the value of the involution on a single strand  of $A_m$ near $x=0$.  Furthermore, an arbitrary choice of this value among the strands of $A_{-m}$ may always be extended to a ruling without switches.  

Now, given a ruling $\rho$ of $A_m A_{-m}$ consider the front diagram, $F_\rho$, arising from resolving the switches of $\rho$ into pairs of horizontal arcs as 
\[
\fig{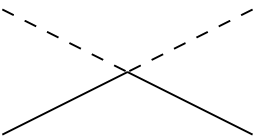} \rightarrow \fig{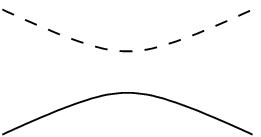}.
\]  $\rho$ gives rise to a normal ruling of $F_\rho$ without switches, and the normality condition forces that 
\[
F_\rho = (A_{i_1}\cdots A_{i_\ell} )(A_{-i_\ell}\cdots A_{-i_1} ), \,\,\, i_1 + \cdots + i_{\ell} = m
\]
with the induced ruling on $F_\rho$ pairing the factors on the left with those on the right in opposite order.  Conversely, any such choice of decomposition $m = i_1 + \cdots + i_{\ell}, i_1,\ldots,i_\ell >0$  and  switchless rulings for $ A_{i_j}A_{-i_j}$, $1 \leq j \leq \ell$, arises in this way from a unique ruling of $A_mA_{-m}$.  The terms in the decomposition $m = i_1 + \cdots + i_{\ell}$ may be reordered to give a partition $\lambda \vdash m$.  In the statement of the lemma, the first term in the sum is the number of ways to rearrange the parts of $\lambda$ to produce 
$ (i_1, \ldots, i_{\ell})$ and the second term $( 1^{m_1} 2^{m_2} \cdots r^{m_r})$ accounts for the choices of switchless rulings.  Finally, the number of switches in a ruling described by this data is $2(\ell(\lambda) - 1)$ which explains the power of $z$.

\end{proof}

For $m \geq 1$ we introduce the notation 
\[
\langle m \rangle = R^2_{A_m A_{-m}}(z),
\]
 and in accordance with Lemma ~\ref{lem:RulComp1} we set
\[
\langle 0 \rangle = z^{-2}.
\]

Next we extend our computations of $R^2$ to products of the $A_m$.  Given partitions  $\lambda = (\lambda_1, \ldots, \lambda_\ell)$ and $\mu = (\mu_1, \ldots, \mu_k)$ we let $A_\lambda$ and $A_{-\mu}$ denote the products
\[
A_\lambda = A_{\lambda_1} \cdots A_{\lambda_{\ell}} \mbox{ and}\, 
A_{-\mu} = A_{-\mu_1} \cdots A_{-\mu_{k}}.
\] 

\begin{theorem} 
\label{the:RulComp2} 
Let $\lambda, \mu \vdash n$ with $\lambda = (\lambda_1, \ldots, \lambda_\ell)$ and $\mu = (\mu_1, \ldots, \mu_k)$.  Denote by $M_{\lambda, \mu}$ the set of $\ell \times k$ matrices with non-negative integer entries such that the entries in the $i$-th row sum to $\lambda_i$ and the entries of the $j$-th column sum to $\mu_j$.  Then,
$$
\displaystyle
R^2_{A_\lambda A_{-\mu}}(z) = z^{2 \ell k - \ell - k} \sum_{(b_{ij}) \in M_{\lambda, \mu} } \prod_{i, j} \langle b_{ij} \rangle.
$$
\end{theorem}

\begin{proof}
Let $\rho$ be a normal ruling of $A_\lambda A_{-\mu}$.  Divide each term $A_{\lambda_i}$ into `blocks' $B_{ij}$ where the block $B_{ij}$ denotes the closure of the portion of $A_{\lambda_i}$ paired with $A_{-\mu_j}$ by $\rho$.  
Distinct blocks can meet only at switches.  The normality condition forces that if two blocks $B_{ij}$ and $B_{ik}$ meet at a switch with $B_{ik}$ containing the upper half of the switching strands and $B_{ij}$ the lower half then $j<k$.  It follows that, after resolving the switches between distinct blocks into horizontal lines, $A_{\lambda_i}$ becomes a product
$A_{b_{ik}} \cdots A_{b_{i2}} A_{b_{i1}}$ with the factor $A_{b_{ij}}$ corresponding to the block $B_{ij}$. (If the block $B_{i,j}$ is empty then we put $b_{ij}= 0$ and treat $A_0$ as an identity.)  Clearly, 
$b_{i1}+b_{i2}+\cdots+ b_{ik} = \lambda_i$.  Since the term $A_{-\mu_j}$ is the union of the closures of the images under $\rho$ of blocks $B_{ij}$, we have also that $b_{1j}+b_{2j}+\cdots+ b_{\ell j} = \mu_j$.  Therefore, $(b_{ij})\in M_{\lambda, \mu} $.  Notice that the ordering of $\rho(B_{ij} )$ along the $z$-axis is likewise forced by the normality condition.  Also, for each $b_{ij}$, $\rho$ gives rise to a normal ruling of the front diagram comprised of the $A_{b_{ij}}$ factor of $A_{\lambda_i}$ and the corresponding portion of $A_{-\mu_j}$ which may be viewed as $A_{-b_{ij}}$.   

Conversely, a decomposition of each $A_{\lambda_i}$ into $A_{b_{ik}} \cdots A_{b_{i2}} A_{b_{i1}}$ and each $A_{-\mu_j}$ into $A_{-b_{\ell j}} \cdots A_{-b_{2j}} A_{-b_{1j}}$ with $(b_{ij}) \in M_{\lambda, \mu}$ together with for each $b_{ij}$ a choice of normal ruling for $A_{b_{ij}}A_{-b_{ij}}$ gives a unique ruling of $A_\lambda A_{-\mu}$.  This justifies the terms in the summation.  Notice that in each $A_{\lambda_i}$ there are $k-1- \#\{j | b_{ij}= 0\}$ switches between distinct blocks.  Similarly,  in each $A_{-\mu_j}$ there are $\ell-1- \#\{i | b_{ij}= 0\}$ switches between the images of distinct blocks.  Combined, these switches account for the $ z^{2 \ell k - \ell - k}$ term in front of the sum and the power of $z$ arising from the product of $\langle b_{ij} \rangle$ with $b_{ij} =0$.  Remaining switches are accounted for in the  $\langle b_{ij} \rangle$ with $b_{ij} \neq 0$ which correspond to the choices of rulings for each  $A_{b_{ij}}A_{-b_{ij}}$.
\end{proof}

\subsection{Distinguishing $A_mA_n$ and $A_n A_m$ using $0$-graded rulings.}

The product on smooth knot types arising from stacking knot diagrams is commutative.  However, the corresponding statement in the Legendrian setting fails to be true.  For instance, using generating family methods Traynor \cite{T1} showed that it is not possible to interchange the two components of the Legendrian link $A_1 A_1$ via a Legendrian isotopy.  Using $0$-graded ruling polynomials we are able to provide a generalization of Traynor's result.

\begin{theorem} 
\label{NonCommute}
 Given non-zero integers $m$ and $n$ it is impossible to interchange the positions of the components of $A_m A_n$  via a Legendrian isotopy.  In particular, if $m \neq n$ then $A_m A_n$  is not Legendrian isotopic to $A_n A_m$.
\end{theorem}

\begin{lemma}  
\label{lem:MasPo} 
Suppose $L_t$, $0 \leq t \leq 1$ is a Legendrian isotopy, so that $L_0 = A_m$ and $L_1$ is a translation of $A_m$ along the $z$-axis.  If $\mu_0$ is a Maslov potential for $L_0$ taking the value $\mu_0 \equiv k \in \Z$ and $\mu_0$ is extended during the isotopy to $\mu_t$, $0 \leq t \leq 1$, then $\mu_1 \equiv k$.
\end{lemma}

\begin{proof}  Consider the Legendrian isotopy $\widetilde{L}_t$ arising from taking the products $L_t A_{-m}$.  Here $A_{-m}$ is placed sufficiently far along the negative $z$-axis to not intersect the fronts $L_t$ at any point during the isotopy.  We equip $\widetilde{L}_t$ with the Maslov potentials $\widetilde{\mu}_t$ where
\[
\widetilde{\mu}_t|_{L_t} = \mu_t \,\,\, \mbox{and} \,\,\, \widetilde{\mu}_t|_{A_{-m}} \equiv k-1
\]
Now using  Theorem ~\ref{RP} and Lemma ~\ref{lem:RulComp1},
\[
R^0_{\widetilde{L}_1, \widetilde{\mu}_1}(z) = R^0_{\widetilde{L}_0, \widetilde{\mu}_0}(z) = R^2_{A_m A_{-m}}(z) \neq 0.
\]
However, if $\mu_1 \neq k$ then 
$R^0_{\widetilde{L}_1, \widetilde{\mu}_1}(z) = 0$.
\end{proof}
\begin{proof}[Proof of Theorem ~\ref{NonCommute}.]
Assume that $L_t, 0\leq t \leq 1,$ is a Legendrian isotopy with $L_0 = A_m A_n$ and $L_1 = A_n A_m$ so that during the course of the isotopy the two components are interchanged.  Without loss of generality we can assume $|m| \geq |n|$ (if not reverse the isotopy) and that $m>0$ and $n <0$ (if not reverse orientations appropriately ).  

Now, consider the isotopy $\widetilde{L}_t$ arising from including an extra component $A_{-m-n }$ far below the other two.  Equip $\widetilde{L}_t$ with a Maslov potential $\mu_0$ so that $ \mu_0|_{A_m} \equiv 1$ and $ \mu_0|_{A_n} = \mu_0|_{A_{-m-n}} = 0$.  $\mu_0$ may be uniquely extended along the isotopy as $\mu_t, 0 \leq t \leq 1$.  According to Lemma ~\ref{lem:MasPo} $\mu_1$ will take these same values on the respective components.  
$(\widetilde{L}_0, \mu_0)$  has $0$-graded rulings (they can be described as in Theorem ~\ref{the:RulComp2}), but $(\widetilde{L}_1, \mu_1)$ does not.  Thus, Theorem ~\ref{RP} gives the contradiction
\[
0 = R^0_{\widetilde{L}_1, \mu_1 }(z) = R^0_{\widetilde{L}_0, \mu_0}(z) \neq 0.
\]

\end{proof}

However, $2$-graded rulings cannot be used to distinguish products of the basic fronts $A_m$, and this will play a crucial role in the proof of Theorem~\ref{mainT}.

\begin{lemma}  
\label{lem:Commute}
If $L_1$ and $L_2$ are products of the basic fronts $A_m$ which differ only in the ordering of factors then $R^2_{L_1}(z) = R^2_{L_2}(z)$.  
\end{lemma}
\begin{proof}  For such a link $L$ suppose that the components of $L$ are precisely
$A_{\alpha_1}, \ldots, A_{\alpha_\ell}$ and $A_{-\beta_1}, \ldots, A_{-\beta_k}$ where $\alpha_1 \geq \ldots \geq \alpha_\ell \geq 1$ and $\beta_1 \geq \ldots \geq \beta_k \geq 1$.  A slight variation of the proof of Theorem ~\ref{the:RulComp2} shows that regardless of the order in which these factors appear we may compute
\[
R_L =  z^{2 \ell k - \ell - k} \sum_{(b_{ij}) \in M_{\alpha, \beta} } \prod_{i, j} \langle b_{ij} \rangle.
\]
Once again, a ruling $\rho$ of $L$ divides each component $A_{\alpha_i}$ into `blocks' $B_{ij}$, $1 \leq j \leq k$ where $\rho( B_{ij}) \subset A_{-b_j}$.  The key observation is that the normality condition still forces the ordering along the $z$-axis  of both the blocks $B_{ij}$ within $A_{\alpha_i}$ as well as their images $\rho(B_{ij})$ within $A_{-\beta_j}$.  Specifically, the ordering of the $B_{ij}$ within $A_{\alpha_i}$ must be as follows.  Cut the $z$-axis just above $A_{\alpha_i}$ and glue the end at $+\infty$ to the end at $-\infty$.  The factors $A_{-\beta_j}$ appear in some order along this now unbroken interval, and the ordering of the $B_{ij}$ within $A_{\alpha_i}$ should be reverse to this.  From here the calculation proceeds as in Theorem ~\ref{the:RulComp2}.
\end{proof}

\section{HOMFLY-PT skein module of the annulus}

Let $R=\Z[a^{\pm1}, z^{\pm1}]$ be the ring of Laurent polynomials in variables $a$ and $z$.  Denote by $\mathcal{L}$ the set of equivalence classes of oriented link diagrams in the annulus up to regular isotopy.  (That is, two diagrams are considered equivalent if they are related via Reidemeister moves of type II or III.)  In addition, let $R\mathcal{L}$ denote the free $R$-module generated by $\mathcal{L}$.

The {\it HOMFLY-PT skein module} of the annulus, $\C$, is the quotient of $R\mathcal{L}$ obtained by imposing the skein relations
\begin{equation*}
\tag{i}\fig{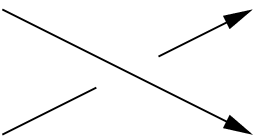} - \fig{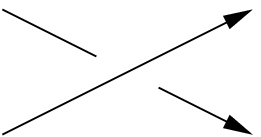} = z \fig{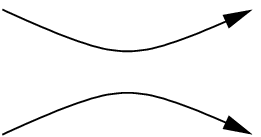}
\end{equation*}
\begin{equation*}
\tag{ii}\fig{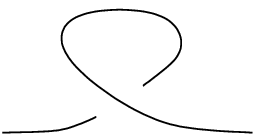} = a \raisebox{-.6ex}{\includegraphics[height=1.2ex]{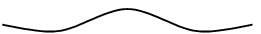}} \,\,\, \mbox{and} \,\,\, \fig{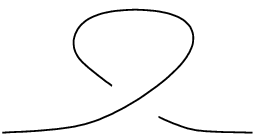} = a^{-1} \raisebox{-.6ex}{\includegraphics[height=1.2ex]{HSR7.eps}},
\end{equation*}
\begin{equation*}
\tag{iii}\displaystyle  D \sqcup \fig{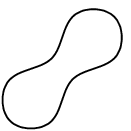} = \frac{a-a^{-1}}{z} D.
\end{equation*}

\begin{remark}
(i) The third relation follows from the first two except in the case when $D$ is the diagram of an empty link.

(ii)  Here we consider usual diagrams of smooth knots and links rather than Legendrian front diagrams.  However, a front diagram may be considered as a usual smooth knot diagram by rounding cusps and hence determines an element of $\C$.

(iii) $\C$ inherits a multiplication from the stacking of diagrams as in Section 2.  In contrast to the Legendrian case, the multiplication is commutative at the level of diagrams.  

(iv)  The diagrams appearing in the skein relations share the same homology class, so $\C$ inherits a grading,
\[
\displaystyle
\C = \bigoplus_{x \in H_1(S^1) } \C_x.
\]
\end{remark}

Turaev introduced the skein module $\C$ in \cite{Tu1}  and proved that $\C$ is free with linear basis, $ \{ A_\lambda A_{-\mu} | \lambda\vdash n_1, \mu\vdash n_2; n_1, n_2 \geq 0  \} $ consisting of monomials in the basic fronts $A_{\pm m}$.  

$\C$ has subalgebras
\[
\C^+ = \bigoplus_{n\geq 0} \C^+_n, \, \C_n^+= \mbox{span}\{A_\lambda | \lambda \vdash n\} 
\]
\[
			\C^- = \bigoplus_{n \geq 0} \C^-_{-n}, \, \C^-_{-n}= \mbox{span}\{A_{-\mu} | \mu \vdash n\}.
\]
satisfying
\begin{equation}
\label{eq:SkeinMod}
\C = \C^+\otimes\C^-.
\end{equation}

Using Turaev's basis, $2$-graded ruling polynomials provide a linear map
\[
\C \rightarrow R, \,\,\, A_\lambda A_{-\mu} \mapsto R^2_{ A_\lambda A_{-\mu} }(z)
\]
which in view of Equation (\ref{eq:SkeinMod}) may be considered as a bilinear form on $\C^+$,
\[
\langle\cdot,\cdot \rangle : \C^+ \times \C^+ \rightarrow R, \,\,\, \langle A_\lambda, A_\mu \rangle = R^2_{A_\lambda A_{-\mu}}(z). 
\]
\begin{remark}  $\langle , \rangle$ is symmetric as reversing the orientation of all components of a Legendrian link will not change the $2$-graded ruling polynomial.  We will see in the next section that $\langle , \rangle$ is actually a positive definite inner product.
\end{remark}

\subsection{Identification of $\C^+$ with the algebra of symmetric functions}

$\C^+$ is a free algebra with unit possessing one generator $A_m$ in each grading degree $m \geq 1$.  Another well known graded algebra with this property is the algebra of symmetric functions $\Lambda$, and in this section we shall fix an isomorphism between them following existing conventions in the literature \cite{AM}, \cite{Lu}, \cite{MM}.  Turaev's geometric basis $A_\lambda$ will be identified with a deformation of the power sum symmetric functions.  As the power sums form a rational basis for $\Lambda$ it will be necessary to begin by enlarging our coefficient ring.  

Let $R'$ denote the smallest subring of rational functions in two variables $a$ and $s$ containing $\Z[a^{\pm 1}, s^{\pm 1}]$ as well as the denominators $s^r-s^{-r}$, $r\geq 1$.  We set $z= s- s^{-1}$ so that $R \subset R'$.   In this section, we consider the HOMFLY-PT skein module $\C_{R'} = R' \otimes_R \C$ over the coefficient ring $R'$ although we will not continue to indicate this with our notation.

Let $\Lambda= \Lambda_{R'} $ denote the algebra of symmetric functions in a countably infinite set of variables ${\bf X}=  \{ x_1, x_2, x_3, \ldots \}$. Here we take coefficients in $R'$.  $\Lambda$ consists of formal polynomials in the $x_i$'s which are unchanged by permuting the variables.  See for instance \cite{Mac} or \cite{St}. A grading,  $\Lambda = \bigoplus_{n\geq 0} \Lambda_n$ arises where $\Lambda_n$ consists of those symmetric functions which are homogeneous of degree $n$ in the $x_i$'s.

\begin{theorem}[\cite{AM}, \cite{Lu}, \cite{MM}] There is an isomorphism  of graded algebras
\[
\begin{array}{ccl} \C^+ & \cong & \Lambda_{R'} \\
Q_\lambda &\leftrightarrow &s_\lambda
\end{array}
\]
where $s_\lambda$ denotes the Schur function and the $Q_\lambda$ satisfy
\begin{equation}
\label{eq:Iso}
A_m = \sum_{\substack{a +b = m-1\\ a,b \geq 0}} (-1)^b s^{a-b} Q_{(a | b)}.
\end{equation}
for $m \geq 1$.  Here, $(a|b)$ denotes the hook partition $(a|b) = (a+1, \underbrace{ 1, \ldots, 1}_{b })$.
\end{theorem}

\begin{remark}
The skein elements $Q_\lambda$ are described in \cite{AM}.  They arise as closures (identify the boundaries) of linear combinations, $E_\lambda$, of link diagrams in the rectangle $[0, 1]\times \R$ with $n$ boundary points on each of $\{0\} \times \R$ and $\{1\} \times \R$ oriented as inputs and outputs respectively.  The $E_\lambda$ are explicitly described in terms of the Young diagram of $\lambda$.  The skein module generated by diagrams of this type in $[0, 1]\times \R$  is one version of the Hecke algebra $\mathcal{H}_n$ (the product here is defined composing diagrams side to side rather than vertically) which specializes to the group algebra of the symmetric group $S_n$ when $s=1$.  The $E_\lambda$ are idempotents which specialize to appropriate multiples of the Young symmetrizers when $s=1$.  

Alternatively, in \cite{Lu} the  $Q_\lambda$ are characterized up to scalars as the eigenvectors of the endomorphism $\varphi : \C^+ \rightarrow \C^+$ defined by adding an extra loop around a diagram
\[
\varphi(X) = \raisebox{-5.4ex}{\includegraphics[height=14.4ex]{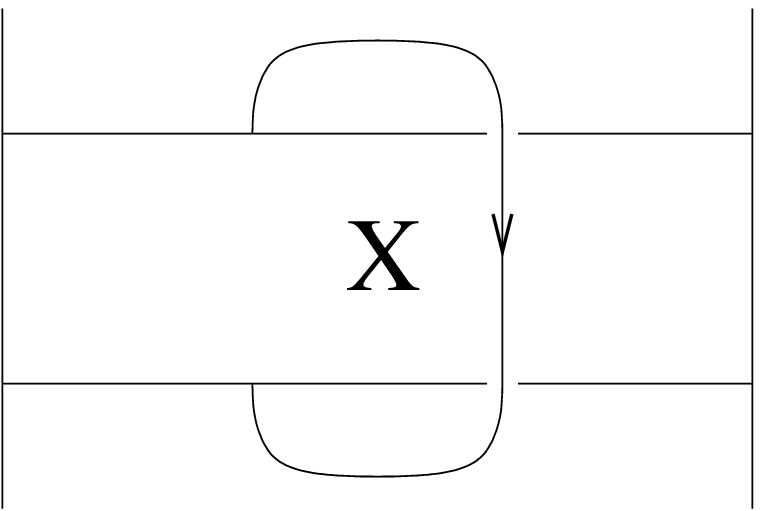}} \,\, .
\]
\cite{Lu} provides as well a skein theoretic proof that the identification of the $Q_\lambda$ with the Schur symmetric functions gives an algebra isomorphism between $\C^+$ and $\Lambda$.  This is remarked in \cite{AM} as a consequence of the fact that the $SU(N)_q$ quantum invariants of links in $\R^3$ with components decorated by irreducible representations $V_\lambda$ may be computed from the HOMFLY-PT polynomial by satelliting each component with the corresponding $Q_\lambda$ and then specializing the variables. 

The relationship of Turaev's basic fronts $A_m$ with the $Q_\lambda$ given in Equation (\ref{eq:Iso}) is found in \cite{MM}.  Under this identification the basic fronts $A_m$ specialize to the power sum functions when $s=1$.  \cite{MM} also contains formulas relating the $A_m$ with other well known bases for $\Lambda$.
 
\end{remark}

\begin{remark}  During the final preparation of this article the author noticed that a seemingly related deformation of the power sum symmetric functions has appeared in the literature on representation theory of Hecke algebras.  The interested reader may wish to make a comparison of the $A_\mu$ described in the present paper with the symmetric functions $q_\mu(x;q)$ appearing in \cite{HLR} keeping in mind that the versions of the Hecke algebra used there and in \cite{MM} differ a bit.  We note that \cite{HLR} contains a computation of the inner product $\left(q_\mu(x;q), q_\lambda(x;q) \right)$ involving a sum of similar nature to the one appearing in our Theorem~\ref{the:RulComp2}, and this result may be related to Theorem \ref{the:product} below.  However, no analog of the variable $z$ is considered in \cite{HLR}, and the proofs seem to be quite different. 
 
\end{remark}

The algebra $\Lambda$ has a standard inner product with respect to which the Schur functions form an orthonomal basis.  Hence, it is natural to define an inner product on $\C^+$ so that the $Q_{\lambda}$ form an orthonormal basis,
\[
(,): \C^+\times \C^+ \rightarrow R', \,\,\,
(Q_\lambda, Q_\mu) := \delta_{\lambda, \mu}.
\]
It turns out that $(,)$ may be interpreted on Turaev's basis $A_{\lambda}$ in terms of ruling polynomials, and in fact agrees with the bilinear form $\langle, \rangle$ defined earlier in this section.

\begin{theorem}  
\label{the:product}
For any partitions $\lambda$ and $\mu$,
\[
(A_\lambda, A_\mu) = R^2_{A_{\lambda}A_{-\mu}}(z) = \langle A_\lambda, A_\mu \rangle.
\]
\end{theorem}

After providing some lemmas we complete this section with the proof of Theorem ~\ref{the:product}.

\begin{lemma} 
\label{lem:m}
For $m\geq 1$,$ (A_m, A_m) = \langle m \rangle.$
\end{lemma}
By definition, we have $\langle m \rangle = \langle A_m, A_m \rangle$ which was computed in Lemma ~\ref{lem:RulComp1}.

\begin{proof}
Consider the generating function $\displaystyle F(t) = z^2 \sum_{m\geq 0} \langle m \rangle t^m$ (we maintain here the convention that $\langle 0 \rangle= z^{-2}$).  Standard calculations with formal power series show that
\[\displaystyle
F(t) = (1 - \sum_{m\geq 1} m z^2 t^m )^{-1}.
\]

Now, introduce the notation 
\[\displaystyle
\left\{ m \right\} := (A_m, A_m) = \sum_{a= 0}^{m-1} s^{2 (2 a - (m-1))}= \frac{s^{2 m} - s^{-2 m} }{s^2-s^{-2} }
\]
and generating function
\[\displaystyle
G(t) = 1 + \sum_{m\geq 1 } z^2 \left\{ m \right\} t^m.
\]
To see that $F(t) = G(t)$ we show that in the product
\[\displaystyle
G(t) \left(1 - \sum_{m\geq 1} m z^2 t^m \right)
\]
\[\displaystyle = 1 + \sum_{m\geq 1}\left( \left[\sum_{k = 1}^{m-1} z^2 \left\{ k \right\}(-(m-k)z^2)\right] + -mz^2 + z^2\left\{ m \right\} \right) t^m
\]
the coefficients of $t^m$ vanish for $m \geq 1$.  After removing a factor of $\displaystyle\frac{z^2}{s^2-s^{-2} } $ the $m$-th coefficient becomes
\[\displaystyle
\left[\sum_{k=1}^{m-1} (s^{2k} - s^{-2k})(k-m)z^2\right] + s^{2m} -s^{-2m} -m(s^2-s^{-2}) =
\]
\[\displaystyle
\left[\sum_{k=1}^{m-1} (k-m)(s^{2k} - s^{-2k})(s^2 - 2 + s^{-2})\right] + s^{2m} -s^{-2m} -m(s^2-s^{-2}). 
\]
Expand the product in the summation.  After collecting terms into pairs and reindexing the summations we have
\[\everymath{\displaystyle}
\begin{array}{rcl} & \left[\sum_{k = 2}^{m} (k-m -1)(s^{2 k} - s^{-2k})\right] & + (s^{2m} - s^{-2m}) + (-m)(s^2-s^{-2}) \\
+ &\left[\sum_{k = 0}^{m-2}(k-m +1)(s^{2 k} - s^{-2k})\right] & + ((m-1) -m +1)(s^{2(m-1)} - s^{-2(m-1)} ) \\
+ & \sum_{k = 1}^{m-1}(-2)(k-m )(s^{2 k} - s^{-2k}) & = 
\end{array}
\]
\[\displaystyle
\sum_{k = 1}^{m-1}\left[ (k-m-1) + (k-m + 1) - 2(k-m) \right] (s^{2 k}- s^{-2k}) = 0.
\]

\end{proof}

To deduce the more general calculation of $(A_\lambda, A_\mu)$ from that of $(A_m,A_m)$ we make use of a coproduct on $\Lambda$.  As described, for instance in \cite{Mac} page 91, one can consider $\Lambda \otimes \Lambda$ as consisting of functions of two countably infinite sets of variables ${\bf X}$ and ${\bf Y}$ which are symmetric with respect to permutations of both ${\bf X}$ and ${\bf Y}$.  Due to the countable number of variables, given $f \in \Lambda$ one may define $\Delta(f) \in \Lambda \otimes \Lambda$ by using a bijection $\mathbb{N}\times \mathbb{N} \cong \mathbb{N}$ to substitute
\[
\Delta(f) ({\bf X}, {\bf Y}) = f({\bf X}, {\bf Y}).
\]

Properties of $\Delta$ which will be important for us include

\begin{itemize}
\item $\Delta$ is an algebra homomorphism.  (In fact $\Lambda$ may be given the structure of a Hopf algebra.)

\item With respect to $(,)$ and the induced inner product on $\Lambda \otimes \Lambda$, $\Delta$ is adjoint to multiplication.  That is, for any $f,g,h \in \Lambda$
\[
(f, g\cdot h) = (\Delta(f), g\otimes h).
\]

\item Coproduct of the Schur functions $Q_{\lambda}$ may be computed as 
\[\displaystyle
\Delta(Q_{\lambda}) = \sum_{\mu, \nu} c_{\mu\nu}^\lambda Q_{\mu} \otimes Q_{\nu}
\]
where $c_{\mu\nu}^\lambda$ are the Littlewood-Richardson coefficients.  

\end{itemize}

Recall that $c_{\mu\nu}^\lambda$ is $0$ unless the Young diagram of $\mu$ is contained in that  of $ \lambda$.  In the latter case $c_{\mu\nu}^\lambda$ is the number of Littlewood-Richardson tableaux of shape $\lambda\setminus \mu$ consisting of $\nu_1$ $1$'s, $\nu_2$ $2$'s, etc.  In turn, such a tableau, $T$, is given by removing those boxes in the Young diagram of $\lambda$ which are contained in $\mu$ and then labeling the remaining boxes with positive integers so that:
\begin{itemize}
\item Rows are weakly increasing from left to right and columns are strictly decreasing from top to bottom, and
\item If a word $w_1w_2\cdots w_n$ is formed from the entries of $T$ by reading each  row from right to left and working top to bottom, then for $k,l\geq 1$ the number of occurrences of $k$ in the truncation $w_1w_2\cdots w_l$ is greater than or equal to the number of occurrences of $k+1$.
\end{itemize}

To simplify the next formula we make the convention that $Q_{(a|b)} =0$ if one of $a$ or $b$ is negative and the other is positive.

\begin{lemma} \label{lem:hook}  For the hook partition $(a|b)\vdash m$ we have
\[\displaystyle
\Delta(Q_{(a|b)}) = \sum_{k=0}^{m-2} \left(\sum_{a'+b' = k} Q_{(a'|b')}\otimes Q_{(a-a'-1|b-b')} + Q_{(a'|b')}\otimes Q_{(a-a'|b-b'-1)} \right)
\]
\[\displaystyle
Q_\emptyset\otimes Q_{(a|b)} + Q_{(a|b)} \otimes Q_\emptyset.
\]
\end{lemma}
\begin{proof}
The convention guarantees that in the summation only $(a'|b')$ with $a' \leq a$ and $b' \leq b$ appear.  When both inequalities are strict there are two Littlewood-Richardson tableaux.  The top row of such a tableau must consist entirely of $1$'s and the left hand column will consist of consecutive integers beginning with either $1$ or $2$.  The first of these accounts for the $Q_{(a-a'|b-b'-1)}$ term and the second for $Q_{(a-a'-1|b-b')}$.  If $a' = a$ or $b'= b$, then there is only one Littlewood-Richardson tableau of shape $(a|b)\setminus (a'|b')$ and according to the convention one of the terms in the sum will correspondingly vanish.  The only remaining possibilities for $\mu$ are $\emptyset$ or $(a|b)$ and these account for the other two terms.
\end{proof}

\begin{proposition} \label{prop:co}
Letting $A_0= z^{-1}$ we have for $m\geq 1$
\[\displaystyle
\Delta(A_m) = z \sum_{i=0}^m A_i \otimes A_{m-i}.
\]
\end{proposition}

\begin{proof}
\[\displaystyle
\Delta(A_m) = \sum_{a+b = m-1}(-1)^b s^{a-b} \Delta( Q_{(a|b)}) = \, \mbox{(Lemma ~\ref{lem:hook} )}
\]
\[\displaystyle
\sum_{a+b = m-1} (-1)^b s^{a-b}\left(  \sum_{k=0}^{m-2} \sum_{a'+b' = k} Q_{(a'|b')}\otimes Q_{(a-a'-1|b-b')} + Q_{(a'|b')}\otimes Q_{(a-a'|b-b'-1)} \right)
\]
\[\displaystyle
+ \sum_{a+b = m-1} Q_\phi\otimes \left((-1)^b s^{a-b} Q_{(a|b)}\right) + \left((-1)^b s^{a-b}Q_{(a|b)}\right) \otimes Q_\phi .
\]
The final two terms are just $1\otimes A_m + A_m\otimes 1 = z(A_0\otimes A_m + A_m \otimes A_0)$.  After putting $c= a-a'$ and $d=b-b'$ the  first term becomes
\[\displaystyle
\sum_{k=0}^{m-2} \sum_{a'+b' = k} \sum_{c+d = m-1-k} (-1)^{b'+d} s^{a'+c-b'-d}\left( Q_{(a'|b')} \otimes Q_{(c-1|d)} +
Q_{(a'|b')} \otimes Q_{(c|d-1)}\right) =
\]
\[\displaystyle
\sum_{k=0}^{m-2} \left( \sum_{a'+ b' =k}(-1)^{b'} s^{a'-b'} Q_{(a'|b')} \right) \otimes s
\left( \sum_{(c-1)+ d =(m-1-k)-1}(-1)^{d}  s^{(c-1)-d} Q_{(c-1|d)} \right) +
\]
\[\displaystyle
\sum_{k=0}^{m-2} \left( \sum_{a'+ b' =k}(-1)^{b'} s^{a'-b'} Q_{(a'|b')} \right) \otimes(- s^{-1})
\left( \sum_{c+ (d-1) =(m-1-k)-1}(-1)^{d-1}  s^{c-(d-1)} Q_{(c|d-1)} \right) =
\]
\[\displaystyle
\sum_{k=0}^{m-2}A_{k+1}\otimes\left(s A_{m-k-1} - s^{-1} A_{m-k-1}\right)= z \sum_{k=1}^{m-1} A_k\otimes A_{m-k}.
\]
\end{proof}

\begin{proof}[Proof of Theorem ~\ref{the:product}.] Let $\lambda = (\lambda_1, \ldots, \lambda_\ell)$, $\mu= (\mu_1, \ldots, \mu_k)$.

Inductively define operators 
\[
D_k: \C^+ \rightarrow (\C^+)^{\otimes k}, D_1 = id, D_{k+1}= (\Delta \otimes (id)^{\otimes k-1} ) \circ D_k.
\]
From the properties of $\Delta$ and Proposition ~\ref{prop:co} we have
\begin{itemize}
\item The $D_k$ are algebra homomorphisms.
\item $(f, g_1 g_2 \cdots g_k) = \left( D_k(f), g_1 \otimes g_2\otimes \cdots \otimes g_k \right)$.
\item Again, letting $A_0 = z^{-1}$, $\displaystyle D_k(A_m) = z^{k-1} \sum_{i_1 + \ldots i_k = m} A_{i_1}\otimes \cdots \otimes A_{i_k}$ where the indices $i_r$ are non-negative integers.
\end{itemize}

Now,
\[\displaystyle
(A_\lambda, A_\mu) = \left( D_k(A_{\lambda_1}\cdots A_{\lambda_\ell}), A_{\mu_1}\otimes\cdots\otimes A_{\mu_k} \right)=
\]
\[\displaystyle
\left( D_k(A_{\lambda_1})\cdots D_k(A_{\lambda_\ell}), A_{\mu_1}\otimes\cdots\otimes A_{\mu_k} \right) =
\]

\begin{multline*}
\displaystyle
\bigg(\left(z^{k-1}\sum_{b_{11} +\cdots+b_{1k} = \lambda_1} A_{b_{11}}\otimes\cdots\otimes A_{b_{1k} }\right)\cdots\left(
 z^{k-1}\sum_{b_{\ell1} +\cdots+b_{\ell k} = \lambda_\ell} A_{b_{\ell1}}\otimes\cdots\otimes A_{b_{\ell k} }\right),
\\
\shoveright{ A_{\mu_1}\otimes\cdots\otimes A_{\mu_k} \bigg)= \\}
\end{multline*}
\begin{equation}\label{eq:A}
\displaystyle
z^{\ell k - \ell} \sum_{(b_{ij}) \in M_{\lambda,\mu} } \prod_{j=1}^k( A_{b_{1j}}\cdots A_{b_{\ell j}}, A_{\mu_j}).
\end{equation}
We are able to restrict the sum to $(b_{ij}) \in M_{\lambda,\mu} $ because the graded components of $\C^+$ are orthogonal with respect to $(,)$.  To conclude, (\ref{eq:A}) becomes
\[\displaystyle
z^{\ell k - \ell} \sum_{(b_{ij}) \in M_{\lambda,\mu} } \prod_{j=1}^k( A_{b_{1j}}\otimes\cdots \otimes A_{b_{\ell j}}, D_\ell(A_{\mu_j}))=
\]
\[\displaystyle
z^{\ell k - \ell} z^{k \ell - k} \sum_{(b_{ij}) \in M_{\lambda,\mu} }\prod_{i,j} (A_{b_{ij}}, A_{b_{ij}}). 
\]
According to Lemma ~\ref{lem:m} and Theorem ~\ref{the:RulComp2} this is equal to $R^2_{A_\lambda A_{-\mu}}$.
\end{proof}

\section{$2$-graded rulings and the Bennequin estimate}

We define the HOMFLY-PT polynomial of a solid torus link $L$ in two steps.  First, using an annular diagram of $L$ and Turaev's basis we have
\[
\begin{array}{ccc} 
\C & \cong & \Z[a^{\pm 1}, z^{\pm 1}, A_{\pm 1}, A_{\pm 2}, \ldots] \\
  \left[L \right]& \leftrightarrow & H_L(a, z, A_i)
\end{array}
\]
$H_L$ is a regular isotopy invariant, and provides an invariant of $L$ as a framed link (assume the framing is blackboard with respect to the projection of $L$ used).  The {\it HOMFLY-PT polynomial} of $L$ is then defined using the normalization
\[
P_L(a,z, A_i) = a^{-w(L) } H_L(a,z, A_i)
\]
where $w(L)$ denotes the {\it writhe} of the diagram $L$.  The writhe is a signed sum of crossings (see Figure ~\ref{fig:writhe}) in the diagram used to compute $H_L$.

Chmutov and Goryunov established the following upper bound in $J^1(S^1)$.
\begin{theorem}[\cite{CG}] \label{chmutov} For any Legendrian link $L \subset J^1(S^1)$,
\[
\textit{tb}(L) + |r(L)| \leq - \mbox{deg}_aP_L.
\]
\end{theorem}

\begin{remark}  A proof of Theorem~\ref{chmutov} is given at the end of Section 6.  
\end{remark}

\begin{figure}
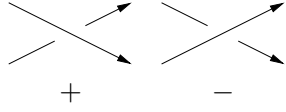

\[\begin{array}{cc}
\includegraphics[height=5.2ex]{HSR1.eps} & \includegraphics[height=5.2ex]{HSR2.eps} \\
+ & - 
\end{array}\]
\caption{A positive crossing and a negative crossing.
}
\label{fig:writhe}
\end{figure}

In $\R^3$ there is a strong connection between an analogous bound and the $2$-graded ruling polynomial \cite{R}.  Namely, $R^2(z)$ is equal to the coefficient of $a^{-\textit{tb}(L)}$ in $P_L$.  (Here we use the convention that the unknot is normalized to $(a-a^{-1})/z$.)  As a consequence, $L$ has a $2$-graded ruling if and only if  $\textit{tb}(L) = - \mbox{deg}_aP_L$.

Analogous results in $J^1(S^1)$ can be obtained provided we specialize the HOMFLY-PT polynomial using the inner product from section 5.  Specifically, for any $L \subset J^1(S^1)$ we let $\widehat{P}_L(a,z) $ be the image of $P_L(a,z, A_i)$ under the $\Z[a^{\pm 1}, z^{\pm 1}]$-module morphism $\C \rightarrow \Z[a^{\pm 1}, z^{\pm 1}]$ defined on Turaev's basis according to
\[
A_{\lambda} A_{-\mu} \mapsto (A_\lambda, A_\mu).
\]

Explicitly, 
\[\displaystyle
P_L(a,z,A_i) = \sum_{\lambda, \mu} c_{\lambda, \mu}(a,z) A_\lambda A_{-\mu} \mapsto
\]
\[
\widehat{P}_L(a,z) = \sum_{\lambda, \mu} c_{\lambda, \mu}(a,z) (A_\lambda, A_{\mu}) 
\displaystyle
=\sum_{\lambda, \mu} c_{\lambda, \mu}(a,z) R^2_{A_\lambda A_{-\mu}}(z).
\]

\begin{theorem} \label{mainT} For any Legendrian link $L \subset J^1(S^1)$ 
\[
R^2_L(z) = \mbox{coefficient of }\, a^{-\textit{tb}(L)} \, \, \mbox{in} \, \, \widehat{P}_L(a,z). 
\]
\end{theorem}

\begin{example}  For the Legendrian $K$ with front diagram pictured in Figure ~\ref{fig:FrontExample} we have
\[
\textit{tb}(K) = 4; \,\,\, P_K(a,z, A_i) = a^{-4}[(1+z^2) A_2A_{-2}] + a^{-6} [z A_1^2A_{-2} + z^2 A_2A_{-2}];
\]
\[\widehat{P}_K(a,z) = a^{-4}(z^4 + 3 z^2 + 2) + a^{-6} (z^4 + 3 z^2); \,\,\, \mbox{and} \,\,R^2_K(z) = z^4 + 3z^2 + 2.
\]
\end{example}

\begin{corollary}  If a Legendrian link  $L \subset J^1(S^1)$ has a $2$-graded ruling, then $\textit{tb}(L)$ is  maximal among knots of the same smooth knot type as $L$. 
\end{corollary}

\begin{corollary}  The $2$-graded ruling polynomial, $R^2(z)$, cannot distinguish Legendrian links in $J^1(S^1)$ with the same smooth knot type and Thurston-Bennequin number.
\end{corollary}

\subsection{Proof of Theorem 6.3}

Let us introduce the notation $B_L(z)$ for the coefficient of $a^{-\textit{tb}(L)}$ in $\widehat{P}_L(a,z)$.  Note that $B_L(z)$ is a Legendrian isotopy invariant.  Using a corresponding specialization of $H_L$,  $B_L(z)$ is given as the coefficient of $a^{c(L)}$ in $\widehat{H}_L(a,z)$ where $c(L)$ is the number of right cusps of $L$.

The proof of Theorem ~\ref{mainT} is based on several lemmas.

\begin{lemma} \label{lem:BaseInd} $R^2_L(z) = B_L(z)$ whenever $L$ is a product of the basic fronts $A_m, m = \pm 1, \pm 2 \ldots$.
\end{lemma}

\begin{proof} 
From Lemma ~\ref{lem:Commute} we know that $R^2_L$ is independent of the ordering of the factors.  This is immediate for $B_L$, so we may assume that $L = A_\lambda A_{-\mu}$.  
Then, $H_L = A_\lambda A_{-\mu}$ and so by the definition of the specialization we have
\[
\widehat{H}_L= (A_\lambda, A_{\mu}) = R^2_{A_\lambda A_{-\mu}}(z).
\]
Since $L$ has no cusps the result follows.
\end{proof}

\begin{lemma} \label{lem:skeinrel}  Both $R^2(z)$ and $B(z)$ satisfy skein relations
\begin{equation*}
\tag{i} \figg{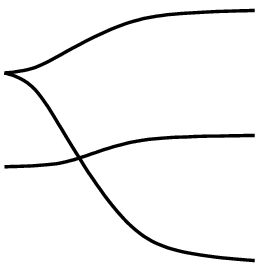} - \figg{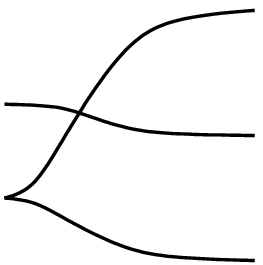} = z\left( \delta_1 \figg{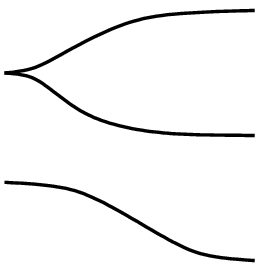} - \delta_2 \figg{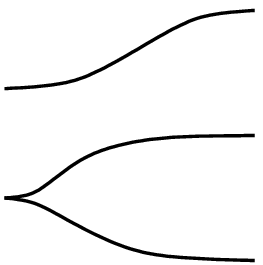}\right)
\end{equation*}
\begin{equation*}
\tag{ii}  \fig{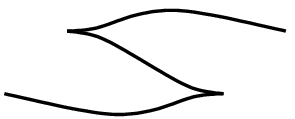} = \fig{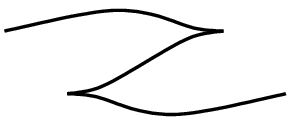} = 0
\end{equation*}
\begin{equation*}
\tag{iii}  K \bigsqcup \fig{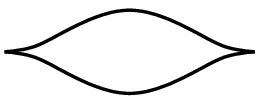} = z^{-1} K.
\end{equation*}
In (i), $\delta_1$ (resp. $\delta_2$) is $1$ when the crossing in the first (resp. second) term on the LHS is positive and $0$ if it is negative.
\end{lemma}

\begin{remark}  Although the orientations are not pictured they are assumed to agree (outside of the pictured portion) in the terms on the LHS of (i).  Whichever term on the RHS has coefficient $\delta_i \neq 0$ is assumed to be oriented in agreement with the terms on the LHS.
\end{remark}

\begin{proof}
The proof is the same as in \cite{R} and will only be sketched here.

To see that $R^2$ satisfies (i), observe that for the two diagrams appearing on the LHS there is a bijection between those rulings where the visible crossing is not switched.  Terms corresponding to these rulings cancel.  Due to the $2$-graded condition only one of the fronts on the LHS can have rulings with the crossing switched.  These remaining rulings are in bijection with the rulings of the term on the RHS with $\delta_i \neq 0$.

For $B_L$, (i) and (iii) follow from the HOMFLY skein relations  and (ii) follows from Theorem ~\ref{chmutov}.
\end{proof}

The proof of Theorem ~\ref{mainT} is then completed by

\begin{lemma} \label{lem:mainL}  A Legendrian isotopy invariant function
\[
\mathcal{F} : \{ \mbox{Annular front diagrams }  \} \rightarrow \Z[a^{\pm 1} , z^{\pm 1}]
\]
satisfying the relations of Lemma ~\ref{lem:skeinrel} is uniquely determined by its values on products of the basic fronts, $A_i, i = \pm 1, \pm 2, \ldots$.
\end{lemma}
 
The proof of Lemma ~\ref{lem:mainL} is by induction on the value of a certain complexity function on front diagrams described in the following subsection.  First we record some additional relations which follow from Lemma ~\ref{lem:skeinrel}.

\begin{lemma} \label{lem:AltRel} A Legendrian isotopy invariant satisfying the relations of Lemma ~\ref{lem:skeinrel} also satisfies
\[
\fig{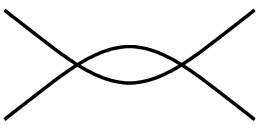} = \fig{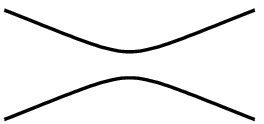} + z\left( \delta_1 \fig{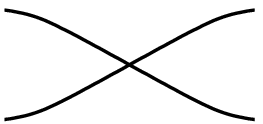} - \delta_2 \fig{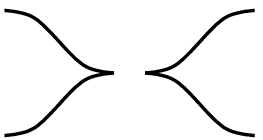} \right).
\]
where  $\delta_1$ (resp. $\delta_2$) is $1$ (resp. $0$) when the crossings in term on the LHS are positive and $0$ (resp. $1$) if they are negative.

\end{lemma}
\begin{proof}
\[
\fig{LemmaSR1.eps} = \fig{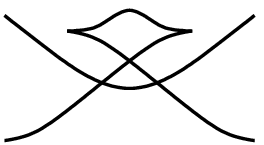} = \fig{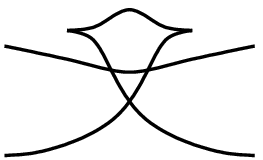} =
\]
\[
\fig{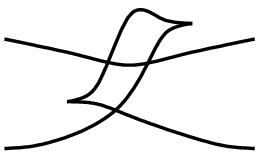} + z\left( \delta_1 \fig{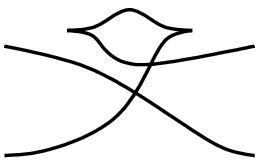} - \delta_2 \fig{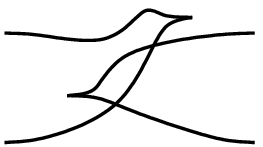} \right) =
\]
\[
\fig{LemmaSR2.eps} + z \left( \delta_1 \fig{LemmaSR3.eps} - \delta_2 \fig{LemmaSR4.eps} \right).
\]
The third equality is Lemma ~\ref{lem:skeinrel}, and the rest are Legendrian isotopies.
\end{proof}

\begin{remark}  Actually, the skein relations  given in Lemma ~\ref{lem:AltRel}  and Lemma ~\ref{lem:skeinrel} (i) are equivalent.
\end{remark}

\subsubsection{Setup for induction}

For non-negative integers $N$ and $M$ we let $\textit{Front}(N,M)$ denote the collection of front diagrams in a rectangle $[0, 1] \times \R$ with $N$ boundary points on $\{0\}\times \R$ and $M$ boundary points on $\{1\}\times \R$. 
 Given front diagrams $f_1 \in \textit{Front}(N_0,N_1)$ and $f_2 \in \textit{Front}(N_1, N_2)$ we may form their product $f_1 f_2 \in \textit{Front}(N_0, N_2)$ by rescaling the first coordinate and then identifying the right boundary of $f_1$ with the left boundary of $f_2$.  (This may involve modifying $f_1$ and $f_2$ a bit near their boundaries so that the boundary points fit together appropriately, but the result is well defined up to Legendrian isotopy.)  

\begin{definition}  A front diagram in $\textit{Front}(N, M)$ is called an {\it elementary tangle} if it contains a single crossing or cusp.
\end{definition}

We adopt the convention of labeling the boundary points of a tangle in $\textit{Front}(N,M)$ as $1, \ldots, N$ and $1,\ldots, M$ from top to bottom. We introduce notations for elementary tangles. $\sigma_m \in \textit{Front}(N,N)$ will denote a crossing between the strands with boundary points labeled $m$ and $m+1$.  
$l_m \in \textit{Front}(N, N+2)$ (resp. $r_m \in \textit{Front}(N+2, N)$) will denote a left (resp. right) cusp where the strands meeting at the cusp are labeled $m$ and $m+1$ at their boundary.  See Figure ~\ref{fig:ElemTan}.

After cutting along vertical lines any $\sigma$-generic annular front diagram  $F$ may be decomposed into a product of elementary tangles,
\[
F = f_1 f_2 \ldots f_n,  \,\,\, f_i \in \textit{Front}(N_{i-1}, N_i), \,\,\, N_0 = N_n.
\]
Here each $f_i$ is some $\sigma_m, l_m,$ or $r_m$.  The factors that appear in such a decomposition of $F$ are unique up to cyclic reordering.  $n$ is called the {\it word length} of $F$.  For our induction we need a slightly more refined measure of the complexity of a front.

\begin{definition}  Given a $\sigma$-generic annular front diagram $F$ as above define the {\it word area} of $F$
\[
\textit{Area}(F) = \sum_{i = 1}^n N_i.
\]
\end{definition}

\begin{figure}
\includegraphics[scale=.75]{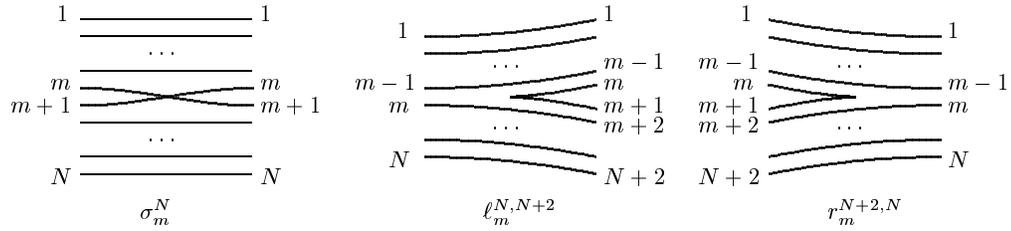}
\caption{ Elementary tangles}
\label{fig:ElemTan}
\end{figure}

\begin{example}  The basic front $A_m$ has word area $m(m-1)$.  The front pictured in Figure ~\ref{fig:AreaEx} has word area 14.   For each $N$ there are fronts with word area $0$ corresponding to the empty product in $\textit{Front}(N,N)$.  These are simply products of the basic fronts $A_1$ and $A_{-1}$. 

\begin{figure}
\includegraphics{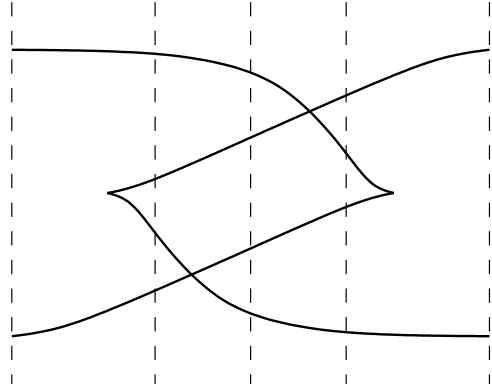}
\caption{An annular front diagram $F$ with $\textit{Area}(F)= 4+4+4+2= 14$.}
\label{fig:AreaEx}
\end{figure}

\end{example}

\begin{proof}[Proof of Lemma ~\ref{lem:mainL}] By induction on $\textit{Area}(F)$.  The base case follows from Lemma ~\ref{lem:BaseInd} since $\textit{Area}(F)=0$ implies $F$ is a product of $A_1$ and $A_{-1}$.  

For the inductive step, given an annular front $F$ we need to show  $\mathcal{F}(F)$ may be evaluated in terms of the values of $\mathcal{F}$ on basic fronts and fronts of lesser word area.  

\noindent {\bf Case 1.}  $F$ has no cusps.

We show that either $F$ is a product of the $A_i$, or we can find a front $F'$ Legendrian isotopic to $F$ so that $\textit{Area}(F) = \textit{Area}(F')$ and part of $F'$ has the form $\fig{LemmaSR1.eps}$.  In the latter case the result follows from Lemma ~\ref{lem:AltRel} as all the front diagrams on the RHS have lesser area than $F'$.

Write $F$ as a word in the $\sigma_m$.  We describe an algorithm to transform $F$ into the desired form using a combination of cyclic permutations and the braid relations 
\[
\sigma_k\sigma_l = \sigma_l \sigma_k, |k-l| \geq 2, \,\,\, \mbox{and} \,\,\, \sigma_i \sigma_{i+1} \sigma_i = \sigma_{i+1} \sigma_i \sigma_{i+1}
\]
both of which correspond to word area preserving Legendrian isotopies.

Assume that we have successfully modified $F$ to a front of the form 
\begin{equation}A_{i_1} \cdots A_{i_r} \overline{F},  r\geq 0 \label{eq:StForm}
\end{equation}
 (product of diagrams in the solid torus). 
If $\overline{F}$ is empty then we have a product of basic fronts and the work is complete.  
 
Else, write $\overline{F} = \sigma_1 \sigma_2 \cdots \sigma_s W, s\geq 0$.  We may always assume that $W$ contains at least one $\sigma_i$ with $i \leq s+1$.  If this is not the case than we could write $\overline{F}= A_{\pm (s+1)} G$ which allows us to absorb the first factor into the product in equation \eqref{eq:StForm} and replace $\overline{F}$ with $G$.

Now, $W$ has the form $\sigma_i W'$ and we proceed as follows:
\begin{itemize}
\item If $i > s+1$, commute  $\sigma_i$ with $\sigma_1 \sigma_2 \cdots \sigma_s$ and cyclicly permute it to get 
\[
\sigma_1 \sigma_2 \cdots \sigma_s \sigma_i W' \rightarrow \sigma_1 \sigma_2 \cdots \sigma_s W' \sigma_i.
\]
Replace $W$ with $W' \sigma_i$ and repeat. 

\item If $i = s+1$, increase $s$ to $s+1$ and repeat the argument with $W$ replaced by $W'$.

\item If $i = s$, then $\overline{F}$ contains $\sigma_s\sigma_s = \fig{LemmaSR1.eps}$ and the algorithm is complete.

\item If $i < s$, then
\[
\begin{array}{cccc}
\sigma_1 \sigma_2 \cdots \sigma_s \sigma_i W'& \rightarrow & \sigma_1 \cdots \sigma_i \sigma_{i+1} \sigma_i \cdots \sigma_s  W' &\rightarrow \\
\sigma_1 \cdots \sigma_{i+1} \sigma_{i} \sigma_{i+1} \cdots \sigma_s  W' &\rightarrow & \sigma_{i+1}\sigma_1 \cdots \sigma_i \sigma_{i+1}  \cdots \sigma_s  W' &\rightarrow \\
\sigma_1 \sigma_2  \cdots \sigma_s  W'  \sigma_{i+1}.  &&&
\end{array}
\]
Now, replace $W$ with $W'\sigma_{i+1}$ and repeat.  
\end{itemize}

It is clear that this procedure cannot loop indefinitely.  $s$ is bounded above, and every time the case $i <s$ occurs the sum of the indices of the $\sigma_i$ occuring in $\overline{F}$ is increased.  

\noindent {\bf Case 2.}  F has cusps.

The following is a slight modification of an argument from \cite{R}.

Note that if the result is known for fronts of lesser area then it is true for a diagram of the form $\cdots l_m\sigma_{m+1} \cdots= \figg{LegSR1.eps}$ if and only if it is true for $\cdots l_{m+1} \sigma_m= \figg{LegSR2.eps}$.  This follows from Lemma ~\ref{lem:skeinrel} since the diagrams appearing on the RHS have smaller area than the two on the LHS.  We will refer to the interchanging of $l_m\sigma_{m+1} $ with $\cdots l_{m+1} \sigma_m$ as a {\it skein move}.  Note that performing a skein move does not change the word area of a front.

This case is dealt with by describing an algorithm which uses a combination of skein moves and Legendrian isotopies to  reduce the word area of $F$ or arrange the front diagram to contain a stabilization \fig{Zig.eps} or a disjoint unknot component \fig{UFO.eps}.  
Whenever skein moves are applied during the algorithm the word area will be such that the inductive hypothesis applies to the corresponding diagrams \figg{LegSR3.eps} and \figg{LegSR4.eps} on the RHS of Lemma ~\ref{lem:skeinrel} so that they may be safely ignored.  In the case that the resulting diagram is stabilized the value of $\mathcal{F}$ is $0$ according to Lemma ~\ref{lem:skeinrel} (ii), and in the case we arrive at an unknot component the value of $\mathcal{F}$ is uniquely determined by Lemma ~\ref{lem:skeinrel} (iii) together with the inductive hypothesis.

The algorithm is nearly identical to Statement A of \cite{R}, but for the reader's convenience we include the argument here.  The reader is also refered to  Figure 1 of \cite{Ng} for an excellent pictoral description of the algorithm.

Begin by writing $F$ as a product of elementary tangles.  It must be the case that there exists a portion of this product which has the form $l_m W r_n$ where $W$ is a word consisting entirely of crossings.  ($F$ has cusps.  Therefore, it must contain both left cusps and right cusps, and one of these left cusps must appear adjacently to a right cusp.)  A cyclic permutation then transforms $F$ into a word of the form $l_m W r_n X$.

Now suppose we are given a word of the form $l_m W r_n X$ where $W$ is a word in the $\sigma_i$ which is written in the form
\[
W = \sigma_{m+1} \sigma_{m+2} \cdots \sigma_{m+s} W', \, \, \mbox{for some} \,\, s\geq 0.
\]

\noindent {\bf Subcase 1.} $W'$ is non-empty.

Then $W' = \sigma_i W''$ and we proceed as follows:
\begin{enumerate}
\item If $i < m-1$, then a Legendrian isotopy commutes $\sigma_i$ past $l_m \sigma_{m+1} \sigma_{m+2} \cdots \sigma_{m+s}$ and when it passes the cusp the word area decreases by $2$.
\item If $i= m-1$, we apply a Legendrian isotopy then a skein move
\[\begin{array}{cccc}
l_m \sigma_{m+1}  \cdots \sigma_{m+s} \sigma_{m-1} W'' \cdots & \rightarrow  & l_m  \sigma_{m-1} \sigma_{m+1}  \cdots \sigma_{m+s} W'' \cdots & \rightarrow \\
l_{m-1}  \sigma_{m} \sigma_{m+1}  \cdots \sigma_{m+s} W'' &&& 
\end{array}
\]
then repeat the argument with $W'$ replaced by $W''$.
\item If $i= m$, then when $s=0$ the front is Legendrian isotopic to a stabilized front, 
\[
\fig{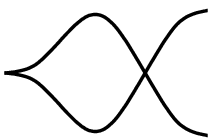} \rightarrow \figg{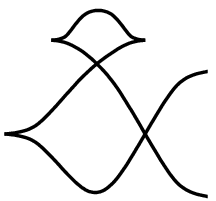} \rightarrow \figg{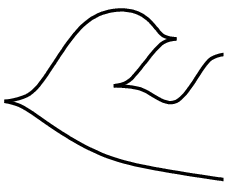}.
\]
When $s>0$ we can apply a Type II Legendrian Reidemeister move to decrease the word area
\[\begin{array}{cccc}
l_m \sigma_{m+1}  \cdots \sigma_{m+s} \sigma_{m} W'' \cdots & \rightarrow  & l_m  \sigma_{m+1} \sigma_m \cdots \sigma_{m+s} W'' \cdots & \rightarrow \\
l_{m+1}  \sigma_{m+2}  \cdots \sigma_{m+s} W'' &&& 
\end{array}
\]
\item If $m < i < m+s$ then we apply a Type III Reidemeister move and subsequently pass a crossing by the left cusp which decreases word area
\[\begin{array}{cccc}
l_m \sigma_{m+1}  \cdots \sigma_{m+s} \sigma_{i} W'' \cdots & \rightarrow  & l_m  \sigma_{m+1}  \cdots \sigma_i \sigma_{i+1} \sigma_i \cdots \sigma_{m+s} W'' \cdots & \rightarrow \\
l_m  \sigma_{m+1}  \cdots \sigma_{i+1} \sigma_{i} \sigma_{i+1} \cdots \sigma_{m+s} W'' \cdots &\rightarrow & \sigma_{i+1} l_m  \sigma_{m+1}  \cdots  \sigma_{m+s} W''.& 
\end{array}
\]  

\item If $i = m+s, s>0$ then we can apply $s$ successive skein moves followed by a Type II Reidemeister move to decrease word area,
\[\begin{array}{cccc}
l_m  \sigma_{m+1} \sigma_{m+2} \cdots \sigma_{m+s} \sigma_{m+s} \cdots & \rightarrow & l_{m+1}  \sigma_{m} \sigma_{m+2} \cdots \sigma_{m+s} \sigma_{m+s}\cdots & \rightarrow \\
l_{m+1} \sigma_{m+2} \cdots \sigma_{m+s} \sigma_{m+s} \sigma_m \cdots & \rightarrow & \ldots & \rightarrow \\
l_{m+s} \sigma_{m+s-1} \sigma_{m+s } \sigma_{m+s-2} \cdots \sigma_{m+1}\sigma_m \cdots  & \rightarrow & l_{m+s-1}  \sigma_{m+s-2} \cdots \sigma_{m+1}\sigma_m&
\end{array}\]

\item If $i = m+s+1$ replace $W'$ with $W''$ and repeat the algorithm.  
\item if $i > m+s+1$ then a Legendrian isotopy commutes $\sigma_i$ past $l_m \sigma_{m+1} \sigma_{m+2} \cdots \sigma_{m+s}$ and when it passes the cusp the word area decreases by $2$.
\end{enumerate}

\noindent {\bf Subcase 2.} $W'$ is empty. 

We are given a word of the form   $l_m  \sigma_{m+1} \sigma_{m+2} \cdots \sigma_{m+s} r_n \ldots$, for some $s \geq 0$
\begin{itemize}
\item If $n < m-1$ or $n > m+s+1$ then we can commute $r_n$ past $l_m  \sigma_{m+1} \sigma_{m+2} \cdots \sigma_{m+s} $ and reduce the word area.
\item If $n = m-1$, then 
\[
l_m  \sigma_{m+1} \sigma_{m+2} \cdots \sigma_{m+s} r_{m-1} \ldots \rightarrow l_m r_{m-1} \sigma_{m+1} \sigma_{m+2} \cdots \sigma_{m+s} \ldots
\]
which is stabilized.
\item If $n = m$ then when $s=0$ we have $l_m r_m \cdots$ which has an unknot component.  When $s >0$ we can apply a Type I Legendrian Reidemeister move and decrease word area.
\item If $m < n < m+s$ then we can apply a type II Reidemeister to decrease word area,
\[\begin{array}{cccc}
l_m \sigma_{m+1}  \cdots \sigma_{m+s} r_n \cdots & \rightarrow  & l_m  \sigma_{m+1}  \cdots \sigma_n \sigma_{n+1} r_n \cdots \sigma_{m+s} \cdots & \rightarrow \\
l_m  \sigma_{m+1}  \cdots r_{n+1} \cdots \sigma_{m+s} \cdots & & & 
\end{array}
\]  
\item If $ n=m +s, s>0$ then because of the presence of $\sigma_{m+s} r_{m+s}$ the front is Legendrian isotopic to a stabilization.
\item If $ n = m+s +1$ we can apply the skein move $s$ times to obtain
\[\begin{array}{cccc}
l_m  \sigma_{m+1} \sigma_{m+2} \cdots \sigma_{m+s} r_{m+s+1} \cdots & \rightarrow & l_{m+1}  \sigma_{m} \sigma_{m+2} \cdots \sigma_{m+s} r_{m+s+1}\cdots & \rightarrow \\
l_{m+1} \sigma_{m+2} \cdots \sigma_{m+s} r_{m+s+1} \sigma_m \cdots & \rightarrow & \ldots & \rightarrow \\
l_{m+s} r_{m+s+1} \sigma_{m+s -1} \sigma_{m+s-2} \cdots \sigma_m \cdots  &&&
\end{array}\]
which is stabilized.
\end{itemize}

\end{proof}

\subsection{Proof of Theorem \ref{chmutov}}

 In \cite{CG} a contactomorphism $J^1(S^1)\cong ST^*(\R^2)$ is used to treat Legendrian links as co-oriented plane curves, and their proof of Theorem~\ref{chmutov} is carried out in this context.  However, the version of $P_L$ used in \cite{CG} differs from ours in a non-trivial manner as the annulus within $J^1(S^1)$ which is used there for link projections differs from ours by a full twist.   
 We conclude by giving a proof of Theorem~\ref{chmutov} matching our conventions.  Our proof uses the front projection perspective and is based on the inductive method used in the proof of Lemma~\ref{lem:mainL}.  This is similar to the approach to Bennequin type inequalities in $\R^3$ appearing in \cite{Ng}.  

First, observe that for a Legendrian link $L \subset J^1(S^1)$ the inequality 
\begin{equation}
\label{eq:Est}
\textit{tb}(L)  + |r(L)| \leq -\deg_a P_L
\end{equation}
is equivalent to
\begin{equation}
\label{eq:altEst}
\deg_a H_L \leq c(L) - |r(L)|.
\end{equation}
Here, $H_L$ is computed using the front projection of $L$, and $c(L)$ denotes the number of right cusps appearing in the front projection.

Observe that this inequality trivially hold for products of the $A_i$ as both sides equal $0$.  We now establish (\ref{eq:altEst}) for a general front diagram $F$ by induction on $\textit{Area}(F)$.  The base case is covered by the previous remark.  

\noindent {\bf Case 1.}  $F$ has no cusps.

As in the proof of Lemma~\ref{lem:mainL}, after a   Legendrian isotopy either $F$ becomes a product of basic fronts, or we can modify $F$ to a front diagram $F'$ containing $\fig{LemmaSR1.eps}$.  In the latter case, use the HOMFLY-PT relations (i) and (ii) to compute $H_{F'}$ as
\begin{equation}
\label{eq:pf}
\fig{LemmaSR1.eps}  = \fig{LemmaSR2.eps} + z\left(\delta_1 \fig{LemmaSR3.eps} +  a^{-1} \delta_2 \fig{LemmaSR4.eps} \right).
\end{equation}
where exactly one of $\delta_1$ and $\delta_2$ is non-zero depending on the orientation of $F'$.   Denote the $3$ front diagrams appearing on the RHS as $F_1$, $F_2$, and $F_3$.   In general, $c(F) = c(F_1) = c(F_2) = c(F_3) -1 $ and $r(F) = r(F_1)$. Also,  if $\delta_1 \neq 0$ (resp. $\delta_2 \neq 0$) then $r(F) = r(F_2)$ (resp. $r(F) = r(F_3)$).  Thus, the inductive hypothesis applies to deduce that both non-zero terms on the RHS of (\ref{eq:pf}) have degree in $a$ less than or equal to $ c(F) - |r(F)|$.

\noindent {\bf Case 2.}  $F$ has cusps.
 
Provided the inductive hypothesis applies to fronts of lesser word area (\ref{eq:altEst}) holds for  a front containing $\figg{LegSR1.eps}$ if and only if it holds for the front $\figg{LegSR2.eps}$  obtained from a skein move.  The proof of Lemma~\ref{lem:mainL} contains an algorithm which makes use of a combination of word area preserving Legendrian isotopies and skein moves to  reduce the word area of $F$ or arrange the front diagram to contain a stabilization \fig{Zig.eps} or a disjoint unknot component \fig{UFO.eps}.  In the latter two cases, proceed as follows:

{\bf 1. } If $F$ contains a stabilization the inductive hypothesis will apply to the front diagram $F_0$ obtained from removing the pair of cusps.  We have $|r(F_0)| = |r(F)| \pm 1$ where the sign depends on the orientation of $F$.  Then, $H_F = H_{F_0}$, and
\[
\deg_aH_{F_0} \leq c(F_0) - |r(F_0)| = (c(F) -1) - (|r(F)| \pm 1) \leq  c(F) - |r(F)|.
\]

{\bf 2.}  If $F$ contains a disjoint unknot component \fig{UFO.eps} then deduce (\ref{eq:altEst}) from the inductive hypothesis and HOMFLY-PT relation (iii).

\end{document}